\documentclass[a4paper]{amsart}
\usepackage{graphicx}
\usepackage{amssymb}
\usepackage{amsmath}
\usepackage{amsthm}
\usepackage{amscd}
\usepackage{xcolor}
\usepackage{enumerate}
\usepackage{array}
\usepackage{caption}
\usepackage{subcaption}
\usepackage{siunitx}
\usepackage{textcomp}
\usepackage[T1]{fontenc}

\newcolumntype{L}[1]{>{\raggedright\let\newline\\\arraybackslash\hspace{0pt}}m{#1}}
\newcolumntype{C}[1]{>{\centering\let\newline\\\arraybackslash\hspace{0pt}}m{#1}}
\newcolumntype{R}[1]{>{\raggedleft\let\newline\\\arraybackslash\hspace{0pt}}m{#1}}

\newtheorem{thm}{Theorem}[section]

\newtheorem{lem}[thm]{Lemma}

\newtheorem{rem}[thm]{Remark}
\theoremstyle{definition}

\newcommand{\R}{\mathbb{R}}

\newcommand{\eps}{\varepsilon}
\newcommand{\ra}{\rightarrow}

\newcommand{\pic}[4]
{
	\begin{figure}[!htbp]
		\begin{center}
			\includegraphics[width=#2\textwidth]{#1}
			\begin{minipage}[c]{0.8\textwidth}
				\begin{center}
					\caption{#3}\label{#4}
				\end{center}
			\end{minipage}
		\end{center}
	\end{figure}
}

\newcommand{\tab}[5]
{
	\begin{center}
		\begin{table}[!htbp]
			\setlength{\extrarowheight}{1pt}
			\begin{tabular}{| r | L{1.5 cm} | L{1.5 cm} | L{1.5 cm} | L{1.5 cm} | L{1.5 cm} |} 
				\hline
				$N$ & $5 \cdot 1$ & $5 \cdot 2$ & $5 \cdot 3$ & $5 \cdot 4$ & $5 \cdot 5$  \\ \hline
				#1
			\end{tabular}
			\newline
			\vspace*{5pt}
			\newline
			\begin{tabular}{| r | L{1.5 cm} | L{1.5 cm} | L{1.5 cm} | L{1.5 cm} | L{1.5 cm} |} 
			\hline
				$N$ & $5 \cdot 6$ & $5 \cdot 7$ & $5 \cdot 8$ & $5 \cdot 9$ & $5 \cdot 10$  \\ \hline
				#2
			\end{tabular}
			\newline
			\vspace*{5pt}
			\newline
			\begin{tabular}{| r | L{2.75 cm} | L{2.75 cm} | L{2.75 cm}|} 
			\hline
				$N$ & $10^2$ & $10^3$ & $10^4$  \\ \hline
				#3
			\end{tabular}
			\vspace{5pt}
			\caption{#4}\label{#5}
		\end{table}
	\end{center}
}
\numberwithin{equation}{section}
\title[Evolution of dislocation density in metallic materials]{On mathematical aspects of evolution of dislocation density in metallic materials}
\author[Natalia Czy\.z{}ewska et. al.]{Natalia Czy\.z{}ewska,  Jan Kusiak, Pawe\l{} Morkisz, Piotr Oprocha${}^\ast$ ,  Maciej Pietrzyk, Pawe\l{} Przyby\l{}owicz, \L{}ukasz Rauch, Danuta Szeliga}
\address[Czy\.z{}ewska, Morkisz, Oprocha, Przyby\l{}owicz]{AGH University of Science and Technology, Faculty of Applied Mathematics, al. Mickiewicza 30, 30-059 Krak\'ow, Poland}
\address[Kusiak, Pietrzyk, Rauch, Szeliga]{AGH University of Science and Technology, Faculty of Metals Engineering and Industrial Computer Science, al. Mickiewicza 30, 30-059 Krak\'ow, Poland}
\begin{document}
\begin{abstract}
This paper deals with the solution of delay differential equations
describing evolution of dislocation density in metallic materials. Hardening,
restoration, and recrystallization characterizing the evolution of dislocation
populations provide the essential equation of the model. The last term
transforms
ordinary differential equation (ODE) into delay differential equation
(DDE) with strong (in general, H\"older) nonlinearity. We prove upper error bounds for
the explicit Euler method, under the assumption that the right-hand
side function is H\"older continuous and monotone which allows us to compare accuracy of other numerical methods in our model (e.g. Runge-Kutta), in particular when explicit formulas for solutions are not known. Finally, we test
the above results in simulations of real industrial process. 
\end{abstract}
\thanks{$\ast$ Corresponding author}
\maketitle	

\tableofcontents

\section{Introduction}
Numerous models of materials developed in the second half of the 20th century use external variables as independent ones \cite{met1}. The model output is a function of some process parameters (e.g., strain, temperature, strain rate), which are external variables and which are grouped in the vector $p$ ($y = y(p)$, where $y$ is the model output). The main drawback of this approach is that it does not properly take into account the history of the considered process. Namely, within these models once the conditions of the process change, the calculated material responses immediately by moving to a new equation of state and the model output is a function of new values of external variables. On the other hand, it was observed experimentally, see for example \cite{met2}, that metallic materials in general show delay in the response to the change in processing conditions. Therefore, the rheological models, which include internal variables as independent parameters, were proposed in the literature. In the internal variable approach (IVM) the model output is a function of time $t$; again of some process parameters (e.g., temperature, strain rate), which we grouped in the vector $p$ and internal variables, which we grouped in the vector $q$: (so now $y = y(t,p,q)$). Since the internal variables remember the state of the material, these models give more realistic description of materials behavior. 

The model with  one internal variable, which is the average dislocation density $\rho$, is usually considered for metallic materials. The model follows fundamental works of Kocks, Mecking, and Estrin \cite{met3,met4}. Main assumptions of this model are repeated briefly below. Since the stress during plastic deformation is governed by the evolution of dislocation populations, 
a competition of storage and annihilation of dislocations, which superimpose in an additive manner, controls a hardening. Thus, the flow stress $\sigma_f$ accounting for softening is proportional to the square root of dislocation density
\begin{equation}\label{eq:sigma}
\sigma_f=a_7+a_6 b\mu \sqrt{\rho},
\end{equation}
where $a_6$ is a material dependent coefficient, $a_7$ is stress due to lattice resistance or solution hardening, $b$ is length of the Burgers vector, and $\mu$ is shear modulus (e.g. see \cite[chapter 3.3]{met1}).

The evolution of dislocation populations is controlled by hardening  ($d\rho/dt=A_1\dot{\varepsilon}$, where $\dot{\varepsilon}$ is the strain rate) and restoration 
($\rho'(t)=-A_2(t)\rho(t)\dot{\varepsilon}(t)$) processes. 
During deformation the dislocation density increases in a monotonic way until the state of saturation is reached. 
%
	However, at elevated temperatures an additional softening mechanism called recrystallization occurs. The term recrystallization is commonly used to describe the replacement of a deformation microstructure by new grains \cite{bibA}.  Processes of phase changes (transformations), which are common in metallic materials, compose nucleation and growth stages. It means that during this process the two phases can coexist.  Recrystallization is classified as a specific type of the transformation, in which the part of the material with increased dislocation density due to deformation is considered an old phase and the part of the material with rebuilt microstructure and free of dislocations is considered a new phase. Two types of the recrystallization can be distinguished, dynamic which occurs during the deformation and static, which occurs after the deformation. The final microstructure and mechanical properties of the alloys are determined, to a large extent, by the recrystallization. The research on the recrystallization dates back to 19th century, and the fast development of the dynamic recrystallization theory was summarized in \cite{bibA}. The most recent research on this process is described in \cite{met5}. A lot of factors have a significant effect on the recrystallization, including the stacking fault energy, the process conditions (temperature, strain rate), the grain size and few other metallurgical parameters. Modeling of recrystallization has been for decades based on the Johnson-Mehl-Avrami-Kolmogorov model, 
 which is based on the external variables only and gives erroneous results when process conditions are changed. In the present work an approach based on the internal variable, which is a dislocation density, was proposed. Since various parts of the material during recrystallization can be in a different state and this process is launched when certain threshold of the accumulated energy is reached, the rate of this process depends on the history of the energy accumulation. The energy accumulated in the material in the form of dislocations from the past acts as the driving force for the current progress of the recrystallization. Similarly, the driving force during static recrystallization depends on the energy accumulated earlier in the material during deformation. This process is launched when certain threshold of the accumulated energy is reached and the rate of this process depends on the history of the energy accumulation. In the mathematical description of this phenomenon a delay differential equation is a natural tool. This approach appeared first in \cite{Dav}. For more details, see Chapter 3.3 in \cite{met1}.


From the above reasoning, it turns out that the evolution of dislocation populations accounting for hardening, recovery, and recrystallization is given by
	\begin{equation}
	\label{RD_0}
	\rho'(t)=A_1(t)\cdot\dot{\varepsilon}(t)-A_2(t)\cdot\rho(t)\cdot\dot{\varepsilon}(t)^{1-a_9}-A_3(t)\cdot(\rho(t))^{a_8}\cdot\mathcal{R}(t-t_{cr})\quad t\geq 0,
	\end{equation}
where as before $t$ is time, $\dot{\varepsilon}$ is the strain rate, $A_1,A_2,A_3$ are model parameters (sometimes time independent, but in most of real world cases dependent on other process parameters, such as $t$, $\dot{\varepsilon}$ etc.),
and $a_8,a_9\in [0,1]$ are additional model coefficients.
The function $\mathcal{R}$ is responsible for the delay in the response to the change in processing conditions, and in the most of practical considerations it is enough to consider 
	\begin{equation}
\mathcal{R}(s)=\mathbf{1}_{(0,+\infty)}(s)\cdot\rho(s).\label{defRsimple}
	\end{equation}
In what follows, we will always use $\mathcal{R}$ in the form \eqref{defRsimple} in \eqref{RD_0}. In what follows, by a solution of \eqref{RD_0} we mean any continuous function $\rho$, which we assume $C^1$ everywhere with the only possible exception at the point $t_{cr}$, where one-sided derivatives may disagree.

The rate of hardening is inversely proportional to the length of the Burgers vector $b$ and the free path for dislocations $l$, that is $A_1 = 1/(bl)$ when $\dot{\varepsilon} > 0$ and $A_1 = 0$ when $\dot{\varepsilon} = 0$. The recovery and the recrystallization are temperature dependent processes following the Aarhenius law \cite{met6}.  The average free path for dislocations $l$, the self-diffusion parameter $A_2$, and the grain boundary mobility $A_3$ in equation \eqref{RD_0} are calculated as
\begin{align}\label{form:l}
				& l =  \left\{ \begin{array}{ll}
				a_1 Z^{-a_{13}}, \quad &\textrm{ when } \dot{\varepsilon} > 0,\\
				0, \quad &\textrm{ when } \dot{\varepsilon} = 0,
				\end{array} \right.
\end{align}
\begin{eqnarray}
A_2&=&a_2 \exp \left(\frac{-a_3}{RT}\right),\\
A_3&=&a_4 \frac{\mu b^2}{2 D} \exp \left(\frac{-a_5}{RT}\right),
\end{eqnarray}
where: $a_2$ is self-diffusion coefficient, $a_3$ is activation energy for self-diffusion, $a_4$ is coefficient of the grain boundary mobility, $a_5$ is activation energy for grain boundary mobility, $D$ is austenite grain size, $Z = \dot{\varepsilon}\exp(Q/(RT))$, is the Zener-Hollomon parameter, $Q$ is activation energy for deformation, $R$ is the universal gas constant equal  to $8.314$ J/mol, and $a_1$, $a_{13}$ are auxiliary model coefficients. Note that in real industrial process even $T$ varies in time, so coefficients $A_1,A_2,A_3$ are complicated functions changing in time.

Critical dislocation density for recrystallization is calculated as
\begin{equation}
\rho_{cr}= a_{11} + a_{12} Z^{a_{10}}\label{eq:rhocr}
\end{equation}
where $a_{10}$, $a_{11}$, $a_{12}$ are coefficients, and $t_{cr}$ is the time between beginning of deformation and beginning of recrystallization (i.e. the moment of reaching $\rho_{cr}$).
Note that $\rho_{cr}$ depends on $Z$ which also changes in time. In simplified approach, $\rho_{cr}$ will be constant, but in practice it is not. In any case, we are interested in the first time that this value is reached (curves representing $\rho$ and $\rho_{cr}$ intersect), which is by the definition value of $t_{cr}$.

In real world, measurement of dislocation density during the process is difficult. Fortunately, flow stress can be measured, and it is dependent on the dislocation density, which evolution is given by equation \eqref{RD_0}. Measurements of flow stress from experiments for different materials are presented in Figure~\ref{pic:strain}.
	
\pic{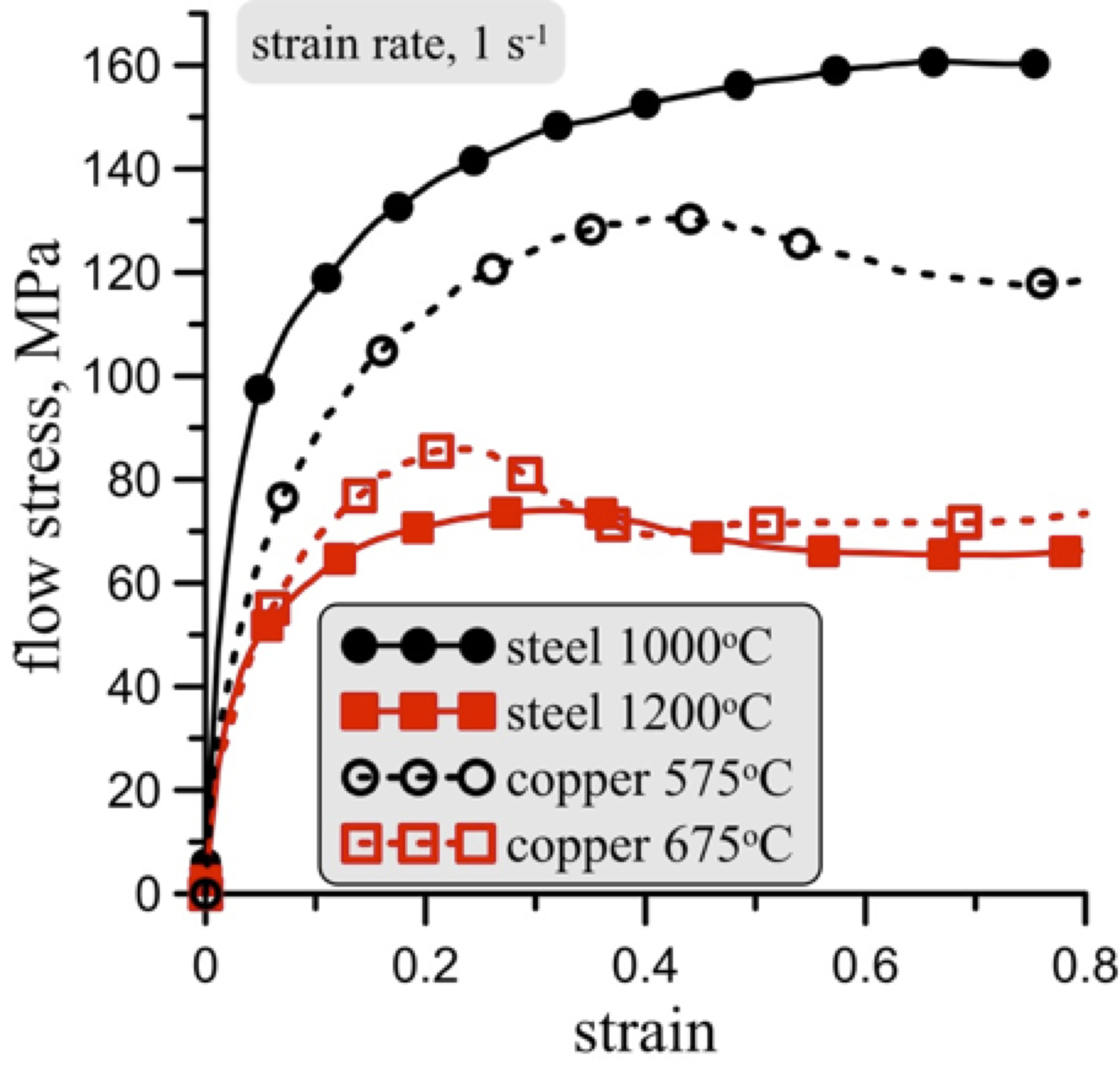}{0.5}{Typical responses of metallic materials subjected to deformation at elevated temperatures, results of the tests for DP steel \cite{Add100} and copper \cite{Add101}}{pic:strain}

Besides numerical simulations, we will perform a detailed theoretical analysis of \eqref{RD_0}, and there are a few good reasons to do that.
First of all, it is hard to find in the literature mathematical tools that can be directly applied to this type of equations, while in recent years some studies of its numerical evolution were undertaken. Classical literature for ordinary differential equations, e.g. \cite{Butcher, Hairer}, assumes some regularity of right-hand side function, commonly Lipschitz condition. Similar assumptions occur for delay differential equations, cf. \cite{Bellen, Smith}. Unfortunately, there is no strict mathematical analysis of the error 
even in the case of standard numerical methods like explicit Euler method, for considered here
nonlinear delay differential equations with a locally H\"older continuous and monotone right-hand side function. In our opinion it is valuable to show that results of these simulations reflect the real behavior of the system. Fortunately, analytic solutions and rigorous formulas can be used for numerical tests on this equation for simplified equations derived from \eqref{RD_0}, especially the cases when coefficients $A_i$ are no longer time (or other process parameters) dependent. As a result of this study we want to ensure that numerical methods, which are accurate at one hand, and have low computational cost at the same time. As we will see, there are good candidates here (as we prove they behave well for simplified models). 

While nowadays there is high popularity in methods of higher order (e.g. Runge-Kutta scheme), they are not suitable for our needs. First of all, observe that in \eqref{RD_0} the right-hand side function is only monotone and locally H\"older continuous, however it is not differentiable at $0$ and it is not even globally Lipschitz continuous (recall that the global Lipschitz condition is usually imposed in the literature). 
Yet another problem in the case of delayed equations, is that in practice in the case of higher order we will need value of delayed function in points not used in mesh of computation. This leads to interpolation of these values, possibly canceling effect of higher order, and making precise error analysis extremely problematic.
Taking all the above reasons into account, we decided to stick with classical Euler scheme whose correctness and suitability we are convinced both numerically and mathematically. In particular, we provide in Theorem \ref{rate_of_conv_expl_Eul} the error bounds for the classical Euler scheme under such irregular assumptions. Moreover, numerical results reported in Section \ref{numerical_solutions} confirmed its good behavior, when applied to the equation \eqref{RD_0} with real-world parameters. For further numerical experiments in real-world setting we refer the reader to our recent paper \cite{NC19}.

The paper is organized as follows. Section~\ref{sec_2} is devoted to existence and uniqueness of solutions of \eqref{RD_0}.  Section~\ref{error} contains error behavior analysis for explicit Euler method with some discussion why we finally chose it for our main numerical experiments. Finally, in Section~\ref{numerical_solutions} some numerical results are given, with simulations for \eqref{RD_0} with real world parameters of selected metallic materials (copper and Dual Phase steel, DP steel for short) at the end.
	
\section{Existence and uniqueness of solutions of some instances of \eqref{RD_0}}\label{sec_2}
In this section we will consider \eqref{RD_0} with some relatively mild additional conditions on time-dependent coefficients $A_1,A_2, A_3$ and $\dot \varepsilon$ (mainly that they are bounded, end extremal values satisfy some relations bonding them together). Before we can prove main results of this section,
we will consider the following auxiliary delay differential equation obtained by simplification of \eqref{RD_0} to the form
\begin{equation}
\label{RD_1}
\rho'(t)=A_1-A_2\cdot \rho(t)-A_3\cdot(\rho(t))^{a_8}\cdot\mathbf{1}_{(t_{cr},+\infty)}(t)\cdot\rho(t-t_{cr}), 
\end{equation}
where $a_8\in [0,1]$ and $A_1,A_2,A_3>0$ are constant. Observe that in \eqref{RD_1}, compared to \eqref{RD_0}, we assume constant strain rate $\dot{\varepsilon}(t)\equiv 1$, and by convention function $\mathcal{R}$ is given by \eqref{defRsimple}. Properties of this simplified equation will allow us to approximate evolution of \eqref{RD_0}.

\subsection{Existence and uniqueness of solution}\label{existance}	
We start with presenting two auxiliary results on \eqref{RD_1}, which will help us to analyze \eqref{RD_0}.
Note that the case $a_8=0$ is very similar to the simple delayed equation $\rho'(t)=A_1-A_2\cdot\rho(t)-A_3\cdot\rho(t-\tau)$ considered in \cite{UF}. Unfortunately, we may not use directly formulas of solutions from there, since in our case of
\eqref{RD_1}, influence of delayed term is also delayed by characteristic function in \eqref{defRsimple}. In \cite{UF} it was pointed out that too large value of $\tau$ with respect to $A,B,C$ can result in unbounded oscillations 
and as a result negative values of $\rho$. In what follows we will see that the condition $\frac{A_3}{A_2}<1$ always prevents it, while as reported in \cite{UF}, cases 
$\frac{A_3}{A_2}\geq 1$ may lead to unstable solutions. The situation in this case is much dependent on the value of $\rho_{cr}$, however.
The analysis of \eqref{RD_1} will lead to analogous conditions on coefficients in \eqref{RD_0}. However as we will see later, our model (with real world parameters)
will satisfy these assumptions.
The following result is an adaptation of the proof of Theorem~3.2. in \cite{Smith} to delay differential equations \eqref{RD_1}. The argument is standard, however, we present it for the reader's convenience.
\begin{lem}\label{lem:bounded}
Let $\rho\colon [0,\sigma)\to \R$ be a noncontinuable solutions of delay differential equations \eqref{RD_1} and assume that $\sigma<+\infty$.
Then $\lim_{t\to \sigma^{-}}|\rho(t)|=+\infty$.
\end{lem}
\begin{proof}
There is $j\geq 0$ such that $jt_{cr}<\sigma\leq (j+1)t_{cr}$. But then we can view $\rho$ as a noncontinuable solution of the ODE defined for $0\leq t<\sigma+t_{cr}$:
$$
	x'(t)=A_1(t)\cdot\dot{\varepsilon}(t)-A_2(t)\cdot x(t)\cdot\dot{\varepsilon}(t)^{1-a_9}-A_3(t)\cdot(x(t))^{a_8}\cdot\mathbf{1}_{(0,+\infty)}(t)\cdot\rho(t-t_{cr}).
$$
If $\rho$ was bounded, then by standard argument for ODEs (e.g. see \cite[Theorem~2.1]{Hale}) $\rho$ can be continued beyond $\sigma$ which is a contradiction.
\end{proof}
\begin{lem}\label{thm:exists:a80}
Assume that $\frac{A_3}{A_2}<1$ and $a_8=0$.
The solutions of delay differential equations \eqref{RD_1} with the initial-value condition
$$
0\leq \rho(0)=\rho_0<\rho_{cr}< A_1/A_2; \quad \rho(t)=\rho_0 \text{ for }t<0
$$
exist for any $t\geq 0$ and are bounded by $[0,A_1/A_2]$.
\end{lem}

\begin{proof}
It is easy to verify that for $t\leq t_{cr}$ the solution $\rho(t)$ exists, is increasing and contained in the interval $(0,\rho_{cr})$. After reaching $t=t_{cr}$, discontinuity in the vector field disappears, and \eqref{RD_1} becomes standard delay differential equation with continuous initial condition, defined by solution of \eqref{RD_1} on the interval $[0,t_{cr}]$
$$
\rho'(t)=A_1 -A_2\rho(t)-A_3 \rho(t-t_{cr}).
$$
In the case that there is a solution that cannot be continued on $\R_+$ it must leave the interval $[0,A_1/A_2]$ first, see Lemma~\ref{lem:bounded}. Denote
$$
\gamma=\sup \{t : \rho(s)\in [0,A_1/A_2] \text{ for all } 0\leq s\leq t\}
$$
and assume that $\gamma<\infty$. It is clear that $\gamma>0$ and there is a decreasing sequence $t_n$, such that $\lim_{n\to \infty}t_n=\gamma$,
$\rho(t_n)$ is well defined (i.e. $t_n$ is in domain of $\rho$) and $\rho(t_n)\not\in [0,A_1/A_2]$.

By definition $\rho(\gamma)\in \{0,A_1/A_2\}$. Let us consider two cases.
\begin{enumerate}
\item Assume first that $\rho(\gamma)=A_1/A_2$. If $\rho(\gamma-t_{cr})>0$
then $\rho'(t)<0$ for $t\in (\gamma-\delta,\gamma+\delta)$ for sufficiently small $\delta$
which is impossible, because for small $\delta$, $\rho'(t)$ is continuous on $(\gamma,\gamma+\delta)$
contradicting the choice of sequence $t_n$.

Let us assume now that $\rho(\gamma-t_{cr})=0$.
This implies that there exists $\delta>0$ such that for $t\in (\gamma-\delta,\gamma+\delta)$ function $r$ defined by
$r(t)=A_1/A_2$ for $t\in [\gamma,\gamma+\delta)$ and $r(t)=\rho(t)$ for $t<\gamma$
is continuous. But then, the function $x(t)=\rho(t)-r(t)$ is a solution of the ODE with continuous vector field defined by
$x'(t)=-A_2x(t)$ for $t\in (\gamma-\delta,\gamma]$ and $x'(t)=-A_2x(t)-A_3 \rho(t-t_{cr})$ for $t\in (\gamma,\gamma+\delta)$, with initial condition $x(\gamma)=0$.
It is clear that $x$ as a solution of that equation must satisfy $x(t)\leq 0$ because for $x>0$ the vector field is negative.
Therefore, for $t\in (\gamma,\gamma+\delta)$ we have $\rho(t)\leq \frac{A_1}{A_2}\leq \rho(\gamma)$.
This is in contradiction with the choice of sequence $t_n$.

\item Assume next that $\rho(\gamma)=0$. There is $\delta$ such that
for $t\in (\gamma-\delta,\gamma+\delta)$ we have $\frac{A_3}{A_2}+\frac{A_2}{A_1} |\rho(t)|<1$.
But then, for all these $t$ we have a lower bound for the values of the vector field
$$
\rho'(t) \geq A_1 -|\rho(t)|A_2-A_3 \frac{A_1}{A_2}\geq A_1\left(1-\frac{A_2}{A_1} |\rho(t)|- \frac{A_3}{A_2}\right)>0
$$
showing that $\rho$ is an increasing function on the interval $(\gamma-\delta,\gamma+\delta)$. A contradiction again.
\end{enumerate}
Indeed, the solution of \eqref{RD_1} is bounded and contained in $[0,A_1/A_2]$, which then implies that it can also be continued onto $\R_+$,
see Lemma~\ref{lem:bounded}.
\end{proof}

While we state the following result for $a_8>0$ in practice we will be interested only in $a_8\in [0,1]$. The proof is standard, we leave details to the reader.

\begin{lem}\label{thm:exists:a8+}
	Assume that $\frac{A_3}{A_2}<1$ and $a_8>0$.
	The solutions of delay differential equations \eqref{RD_1}	and with initial-value condition
	$$
	0\leq \rho(0)=\rho_0<\rho_{cr}< A_1/A_2; \quad \rho(t)=\rho_0 \text{ for }t<0
	$$
	exist for any $t\geq 0$ and are bounded by $[0,A_1/A_2]$.
\end{lem}

\begin{rem}
In the following theorem we may replace $m\alpha_2$ by $\inf \dot{\eps}A_2$ and $M\beta_1$ by $\sup \dot{\eps}A_1$. This way it can be applied to a slightly larger class of equations.
\end{rem}

\begin{rem}
	In practice, the value of temperature $T(t)$ is much higher than $0$ and by physical constraints also bounded from the above. So if $\dot\eps(t)>m>0$ then $Z(t)\subset [a,b]\subset (0,+\infty)$ and as a consequence all the coefficients $A_i(t)$ are bounded and separated from zero.
\end{rem}

\begin{thm}\label{thm:exists}
Assume that there are positive constants $0<\alpha_i\leq \beta_i$ and $0<m\leq M$ such that coefficients $A_i(t)\in [\alpha_i,\beta_i]$
(for each $t\geq 0$; this takes into account other process parameters 
that these coefficients are dependent)
and $\dot \eps(t) \in [m,M]$ for every $t\geq 0$.
Additionally assume that $\frac{\alpha_3}{m\alpha_2}<1$ and either $a_8>0$ or $\frac{\beta_3}{m\beta_2}<1$.
Then the solutions of delay differential equations \eqref{RD_0} with the initial-value condition
$$
0\leq \rho(0)=\rho_0<\rho_{cr}< \frac{M\beta_1}{m\alpha_2}; \quad \rho(t)=\rho_0 \text{ for }t<0
$$
exist for any $t\geq 0$ and are bounded by $[0,\frac{M\beta_1}{m\alpha_2}]$ and are unique.	
\end{thm}
\begin{proof}
Consider the following equations with constant coefficients:
\begin{equation}
	z'(t)=M\beta_1-m\alpha_2z(t)-\alpha_3 \mathbf{1}_{[t_{cr},\infty)}(t)z(t-t_{cr})
	\label{eq:z1}
	\end{equation}
	and with initial-value condition
	$z(t)=\rho_0$ for all $t\leq 0$ and
	\begin{equation}
	w'(t)=m\alpha_1-M^{1-a_9}\beta_2w(t)-\beta_3 \mathbf{1}_{[t_{cr},\infty)}(t)w(t-t_{cr})
	\label{eq:w1}
	\end{equation}
	and with initial-value condition
	$w(t)=0$ for all $t\leq 0$ where $t_{cr}$ is provided solution of \eqref{RD_0}, provided it exists. In the other case we omit delay term in \eqref{eq:z1} and \eqref{eq:w1}.
	By Lemmas~\ref{thm:exists:a80} and \ref{thm:exists:a8+} solutions of \eqref{eq:z1} and \eqref{eq:w1} exist for every $t>0$ and $z(t)< \frac{M\beta_1}{m\alpha_2}$ while $w(t)\geq 0$.
	In fact, $w(t)\geq 0$ when $a_8>0$ despite of relations between other coefficients.	
Simple calculations yield that for $t<t_{cr}$ we have $z'(t)-\rho'(t)\leq 0$ and $\rho'(t)-w'(t)\geq 0$ which implies that $\rho(t)$ exists for $t\in [0,t_{cr}]$ and $w(t)\leq \rho(t)\leq z(t)$ and this inequality can be recursively extended onto further intervals $[nt_{cr}, (n+1)t_{cr}]$ which completes the proof of boundedness of solutions.

Consider time interval $[nt_{cr},(n+1)t_{cr})$ for $n=0,1,\ldots$. We may view $\eqref{RD_2}$ on each of these intervals as ODE. Let
$$
f(t,x)=	A_1(t)\cdot\dot{\varepsilon}(t)-A_2(t)\cdot x\cdot\dot{\varepsilon}(t)^{1-a_9}-A_3(t)\cdot x^{a_8}\cdot\mathcal{R}(t-t_{cr})
$$
where on each of the above intervals delay term $\mathcal{R}(t-t_{cr})$ can be regarded as a function of $t$ but independent of solution. In fact we can view $\mathcal{R}(t-t_{cr})$
as a function defined for all $t\in \R$ by putting $\mathcal{R}(t-t_{cr})=\mathcal{R}(nt_{cr})$ for all $t>(n+1)t_{cr}$.
This way we may regard \eqref{RD_0} as associated ODE
\begin{equation}
	\label{RD_2}
	\rho'(t)=f(t,\rho(t))
\end{equation}
with initial condition $\rho(nt_{cr}):=\lim_{s\to nt_{cr}^-}\rho(s)$, since values $\rho(s)$ for $s\in [(n-1)t_{cr},nt_{cr})$ have already been determined. It is obvious that $\rho(t_{cr})\geq 0$. Note that there are $\delta,\alpha>0$ such that for $x\in [0,\delta]$ we have $f(t,x)>\alpha$, so in particular $\rho(t)$ is bounded away from $0$, provided it is defined. We already know that there is a solution $\rho$ of \eqref{RD_0} (so also \eqref{RD_2}) in $[nt_{cr},(n+1)t_{cr})$ and assume that $\bar\rho$ is another solution in $[t_{cr},2 t_{cr})$, but with the same initial value as $\rho$, i.e., $\bar\rho(t_{cr})=\rho(t_{cr})$. Then we have
\begin{equation}
	\frac{d}{dt}(\rho(t)-\bar\rho(t))^2=2(\rho(t)-\bar\rho(t))(f(t,\rho(t))-f(t,\bar\rho(t)))\leq 0,
\end{equation}
since for all $t$ the function $(0,+\infty)\ni x\mapsto f(t,x)$ is nonincreasing (recall $\rho(t)\geq 0$ for all $t\geq 0$), hence
\begin{equation}
	0\leq(\rho(t)-\bar\rho(t))^2\leq (\rho(nt_{cr})-\bar\rho(nt_{cr}))^2=0.
\end{equation}
Repeating the above arguments inductively on consecutive intervals $[nt_{cr},(n+1)t_{cr})$ we complete the proof.
\end{proof}

\begin{rem}
In practical applications, the condition $\dot \eps(t)\geq m>0$ will not be usually satisfied. The reason is that inside the metallic material usually it will take some time
to observe $\dot \eps(t)>0$. However after some time, say $T_0$ we will have $\dot \eps(t)>m$ for all $t>T_0$ and at the same time $\rho$ will not diverge too much from $\rho_0$.
The reader may check that in these cases, statements of Theorem~\ref{thm:exists} are still valid.

The same reasoning can be applied to other coefficients.
\end{rem}


\subsection{Some solutions of the toy model \eqref{RD_1} and existence of $t_{cr}$}\label{formulas:timindep:01}
For the equation (\ref{RD_1}) we can give explicit formula for solution in the case when $a_8\in\{0,1\}$. For the fractional values of $a_8$ it is rather hard to provide analytic formulas.
On the other hand, we may view the above two cases of $a_8$ as bounds for intermediate values. The main utility of these formulas, is that they can be used to strict control of error in preliminary numerical experiments.
As usual, let us assume that
\begin{equation}
\label{cond_r_cr}
0\leq \rho_0<\rho_{cr}<A_1/A_2.
\end{equation}
Under the assumption above the solution $\rho$ attains the critical value $\rho_{cr}$ in finite time $t_{cr}$, and it is not hard to check that
\begin{equation}
\label{t_cr_tindep}
	t_{cr}=\frac{1}{A_2}\ln\Bigl(\frac{\rho_0-(A_1/A_2)}{\rho_{cr}-(A_1/A_2)}\Bigr).
\end{equation} 
Namely, for $t\in [0,t_{cr}]$ the equation \eqref{RD_1} is reduced to a simple linear equation with the solution
\begin{equation}
	\label{lin_eq_sol}
	\rho(t)=\frac{A_1}{A_2}+\Bigl(\rho_0-\frac{A_1}{A_2}\Bigr)e^{-A_2t},
\end{equation}
 which is a strictly increasing function and $\rho(t_{cr})=\rho_{cr}$. 
In the intervals $[n t_{cr},(n+1)t_{cr}]$, $n\in\mathbb{N}$, we solve the equation (\ref{RD_1}) recursively as follows. Let us denote by $\phi_{n-1}$ the solution $\rho$ in the interval $t \in [(n-1)t_{cr},nt_{cr}]$ ($\phi_0$ in $[0,t_{cr}]$ is given by (\ref{lin_eq_sol})). We have two cases:
 \begin{enumerate}[(i)]
 \item $\mathbf{a_8=0}$: 
The solution with the initial value $\rho(nt_{cr})=\phi_{n-1}(nt_{cr})$ is
\begin{equation}
\label{RD_sol_n0}
	\rho(t)=e^{-A_2(t-nt_{cr})}\cdot\Biggl(\phi_{n-1}(nt_{cr})+\int\limits_{nt_{cr}}^te^{A_2(s-n t_{cr})}\cdot\Bigl(A_1-A_3\cdot\phi_{n-1}(s-t_{cr})\Bigr)ds\Biggr),
\end{equation}
for $t\in [nt_{cr},(n+1)t_{cr}]$.
 \item $\mathbf{a_8=1}$: 
The initial value $\rho(nt_{cr})=\phi_{n-1}(nt_{cr})$ leads to the solution
\begin{equation}
\label{RD_sol_n1}
	\rho(t)=e^{-\int\limits_{nt_{cr}}^t p_{n-1}(s)ds}\cdot\Biggl(\phi_{n-1}(nt_{cr})+A_1\cdot\int\limits_{nt_{cr}}^t e^{\int\limits_{nt_{cr}}^s p_{n-1}(u)du}ds\Biggr),
\end{equation}
where
\begin{equation}
	p_{n-1}(t)=A_2+A_3\cdot\phi_{n-1}(t-t_{cr}),  \quad t\in [nt_{cr},(n+1)t_{cr}].
\end{equation}
\end{enumerate}
Note that both solutions are given in integral form, which most likely is impossible to present as explicit functions in the case $a_8=1$, since we have doubly exponential terms under integral.
For the case $a_8=0$ it seems possible to provide some formulas (similarly to \cite{UF}), however their complexity increases rapidly with multiplies of $t_{cr}$, mainly because vector field is discontinuous. 

On the other hand, equations \eqref{RD_sol_n0} and \eqref{RD_sol_n1} can be treated with numerical integration with rigorous control of numerical errors. This way	we can accurately estimate numerical errors of numerical solutions of these equations (e.g. by explicit  Euler method). This gives us a chance for rigorous comparison of various numerical methods
	for solving equations of type \eqref{RD_1}, possibly ensuring similar behavior of numerical approximations of its further generalizations.
\begin{rem}\label{rem_Bern} \rm The assumption that $A_1>0$ seems to be crucial in order to have a nontrivial problem, since, in the case when $A_1=0$, we get by (\ref{cond_r_cr}) that $\rho_0=\rho_{cr}=0$ and $t_{cr}=0$. Moreover, the equation (\ref{RD_1}) becomes the following Bernoulli equation
\begin{equation}
	\rho'(t)=-A_2\cdot\rho(t)-A_3\cdot(\rho(t))^{a_8+1},\quad t>0,
\end{equation}
which, under the initial value $\rho_0=0$, has the following trivial solution $\rho(t)=0$ for all $t\in [0,+\infty)$ and all $a_8\in [0,1]$. By Theorem \ref{thm:exists:a8+} all solutions of (\ref{RD_1}) tend to the zero solution when $A_1\to 0+$. Hence, only the case when $A_1>0$ is of practical  interest. It is worth mentioning, that it is always the case in considered models.
\end{rem}	

Unfortunately in applications we cannot assume that coefficients $A_i$ are time independent.
Then the question arise to which extent the new equation is similar. The first step will be to show that critical time $t_{cr}$, under certain assumptions on $A_1,A_2$, always exists also in that case. Assume that assumptions of Theorem~\ref{thm:exists} are satisfied, in particular coefficients $A_i$ as well as $\dot \eps$ are bounded, i.e.
$A_i(t)\in [\alpha_i,\beta_i]$ and $\dot \eps(t) \in [m,M]$ for every $t\geq 0$.

Consider time-dependent version of \eqref{RD_1} before reaching $t_{cr}$ derived from \eqref{RD_0} with ``smallest possible'' vector fiels, that is the equation (stated for $t\geq 0$)
\begin{equation}
\label{RD_1_td}
\rho'(t)=m A_1(t)-M^{1-a_9} A_2(t)\cdot \rho(t), \quad \rho(0)=\rho_{0}\geq 0.
\end{equation}
Hence
\begin{equation}
\label{lin_eq_td}
\rho(t)=e^{-M^{1-a_9} \int\limits_0^t  A_2(s)ds}\cdot\Bigl(\rho_0+m\int\limits_0^t e^{M^{1-a_9} \int\limits_0^s A_2(u)du}A_1(s)ds\Bigr),
\end{equation}
which is clearly continuous and positive function on $[0,+\infty)$, and since both $A_1,A_2$ are bounded away from zero, we also have
\begin{equation}
\lim\limits_{t\to +\infty}e^{\int\limits_0^t A_2(s)ds}=+\infty, \quad
\lim\limits_{t\to +\infty}\int\limits_0^t e^{\int\limits_0^s A_2(u)du}A_1(s)ds=+\infty.
\end{equation}
Therefore solution of both \eqref{lin_eq_td} and \eqref{RD_0} satisfy 
$$
\limsup\limits_{t\to +\infty}\rho(t)\geq \frac{m}{M^{1-a_9}}\limsup\limits_{t\to +\infty}\frac{ A_1(t)}{A_2(t)}.
$$
In particular, when $\limsup\limits_{t\to +\infty}\frac{A_1(t)}{A_2(t)}>\frac{M^{1-a_9} \rho_{cr}}{m}$ then there exists $t_{cr}>0$ such that the solution $\rho$ of \eqref{RD_1_td} satisfies $\rho(t_{cr})=\rho_{cr}$ (and by continuity we may assume that $t_{cr}$ is smallest among all such times). 

For example, let us consider the equation \eqref{RD_0} with $a_9=0$, time independent but positive $A_1,A_2$, continuous $\dot\varepsilon$ and, as before, assume that $\dot\varepsilon(t)\geq m>0$ for all $t\geq 0$. Hence, we are considering time-dependent version of \eqref{RD_1}, with $A_1(t)=A_1\dot{\varepsilon}(t)$, $A_2(t)=A_2\dot{\varepsilon}(t)$ and $A_3(t)=A_3$. Under the assumption \eqref{cond_r_cr} we have, by the above considerations, that there always exists $t_{cr}>0$, since in that case $\limsup\limits_{t\to +\infty}\rho(t)=A_1/A_2>\rho_{cr}$ . Moreover, it can be shown that 
\begin{equation}
	\label{t_cr_timedep}
		t_{cr} = \varepsilon^{-1}\Bigl(\frac{1}{A_2}\ln\Bigl(\frac{\rho_0-(A_1/A_2)}{\rho_{cr}-(A_1/A_2)}\Bigr)+\varepsilon(0)\Bigr),
\end{equation}
where $\varepsilon^{-1}$ is the inverse function for $\varepsilon$, which exists since $\eps$ is strictly increasing. Note that in the case when $\dot\varepsilon\equiv 1$ we restore from \eqref{t_cr_timedep} the equation \eqref{t_cr_tindep}. Nevertheless, only in this particular case we know the closed formula for $t_{cr}$. In general the nonlinear equation $\rho(t_{cr})=\rho_{cr}$ has to be solved numerically.

As long as $t_{cr}$ is calculated, we can repeat arguments presented earlier in Section~\ref{formulas:timindep:01} and provide
formulas for solutions when $a_8\in\{0,1\}$. As before, for $t\in [n t_{cr},(n+1)t_{cr}]$, $n\in\mathbb{N}$, we solve the equation (\ref{RD_1}) with time-dependent $A_1,A_2,A_3$ recursively. Let us denote by $\phi_{n-1}$ the solution $\rho$ in the interval $t \in [(n-1)t_{cr},nt_{cr}]$, where $\phi_0$ in $[0,t_{cr}]$ is given by 
\begin{equation}
	\label{lin_eq_td2}
		\phi_0(t)=e^{-\int\limits_0^t A_2(s)ds}\cdot\Bigl(\rho_0+\int\limits_0^t e^{\int\limits_0^s A_2(u)du}A_1(s)ds\Bigr).
\end{equation}
 We consider the following two cases:
\begin{enumerate}[(i)]
	\item  $\mathbf{a_8=0}$: Then for $t\in [nt_{cr},(n+1)t_{cr}]$, $n\in\mathbb{N}$, the solution is given by
	\begin{equation}
	\rho(t)=e^{-\int\limits_{nt_{cr}}^t A_2(s)ds}\cdot\Biggl(\phi_{n-1}(nt_{cr})+\int\limits_{nt_{cr}}^te^{\int\limits_{nt_{cr}}^sA_2(u)du}\cdot q_{n-1}(s)ds\Biggr),
	\end{equation}
	where
	\begin{equation}
	q_{n-1}(t)=A_1(t)-A_3(t)\cdot\phi_{n-1}(t-t_{cr}).
	\end{equation}
	\item $\mathbf{a_8=1}$:  Then for $t\in [nt_{cr},(n+1)t_{cr}]$, $n\in\mathbb{N}$, the solution is 
	\begin{equation}
	\label{a81_analytical_sol}
	\rho(t)=e^{-\int\limits_{nt_{cr}}^t p_{n-1}(s)ds}\cdot\Biggl(\phi_{n-1}(nt_{cr})+\int\limits_{nt_{cr}}^t e^{\int\limits_{nt_{cr}}^s p_{n-1}(u)du}\cdot A_1(s)ds\Biggr),
	\end{equation}
	where
	\begin{equation}
	p_{n-1}(t)=A_2(t)+A_3(t)\cdot\phi_{n-1}(t-t_{cr}),  \quad t\in [nt_{cr},(n+1)t_{cr}].
	\end{equation}
\end{enumerate}
Despite, the formulas being slightly more complicated, they can be effectively used within the process of evaluation of  accuracy and correctness of numerical methods
used to solve time-dependent versions of \eqref{RD_1}.
\section{Error analysis of the explicit Euler method}\label{error}
Since in the case when $a_8\in (0,1)$ the exact formulas for the solution of (\ref{RD_0}) are not known, we use the suitable numerical methods to approximate $\rho$ on a finite time interval. We are interested in the error analysis for solutions of \eqref{RD_0} after reaching $t_{cr}$, since delay activates at this point. This approach will allow us to impose some reasonable assumptions on continuity of vector field. It will be visible in assumptions (F1)-(F4) below; see also Remark~\ref{remark_tilde_rho}.

We consider the (general) delay differential equation
\begin{equation}
\label{dde_1}
	z'(t)=f(t,z(t),z(t-t_{cr})), \quad t\geq 0,
\end{equation}
with a given right-hand side function $f\colon [0,+\infty)\times\mathbb{R}\times\mathbb{R}\to\mathbb{R}$ and where $z(t)=\eta\in\mathbb{R}$ for $t\in [-t_{cr},0]$.

For fixed $n\in\mathbb{N}$ the \textit{explicit Euler method} that approximates a solution $z=z(t)$ of \eqref{dde_1} for $t\in [0,(n+1)t_{cr}]$ is defined recursively for subsequent intervals. Namely, let $N\in\mathbb{N}$ and
\begin{displaymath}
	t_k^j=jt_{cr}+kh, \quad k=0,1,\ldots,N, \ j=0,1,\ldots,n,
\end{displaymath} 
 where
\begin{equation}
	h=\frac{t_{cr}}{N}.
\end{equation}
Note that $\{t^j_k\}_{k=0}^N$ is uniform discretization of the subinterval $[jt_{cr},(j+1)t_{cr}]$. Discrete approximation of $z$ in $[0,t_{cr}]$ is defined by
\begin{eqnarray}
	y_0^0&=&\eta,\label{eq:43}\\
	y_{k+1}^0&=&y_k^0+h\cdot f(t_k^0,y_k^0,\eta), \quad k=0,1,\ldots, N-1.\label{eq:44}
\end{eqnarray}
Let us assume that the approximations $y_k^{j-1}\approx z(t_k^{j-1})$, $k=0,1,\ldots,N$, have already been defined in the interval $[(j-1)t_{cr},jt_{ct}]$ (for $j=1$ it was done in \eqref{eq:43} and \eqref{eq:44}). Then for $j=1,2,\ldots,n$ we take
\begin{eqnarray}
\label{expl_euler_1}
	y_0^j&=&y^{j-1}_N,\\
\label{expl_euler_11}	
	y_{k+1}^j&=&y_k^j+h\cdot f(t_k^j,y_k^j,y_k^{j-1}), \quad k=0,1,\ldots, N-1,
\end{eqnarray}
as the approximation of $z$ in $[jt_{cr},(j+1)t_{cr}]$.

In this section we present rigorous analysis of the error of the explicit Euler method under the nonstandard assumptions on the right-hand side function $f$ of the equation \eqref{dde_1}. Namely, we assume that $f$ is monotone and locally H\"older continuous instead of the global Lipschitz continuity. According to the authors knowledge, there is lack of such analysis in the literature (cf. \cite{Bellen, Butcher, Hairer}), since we consider a non-Lipschitz case. 

Let us emphasize once again, that we cannot apply higher order methods under out assumptions (such us Runge-Kutta schemes), since  we do not assume that function $f$ is differentiable. Observe that indeed it is the case of our main equation \eqref{RD_0}, since the right-hand side function of \eqref{RD_0} is not differentiable at $0$ when $a_8\in (0,1)$.

For the right-hand side function $f:[0,+\infty)\times\mathbb{R}\times\mathbb{R}\to\mathbb{R}$ in the equation (\ref{dde_1}) we impose the following assumptions:
\begin{itemize}
	\item [(F1)] $f\in C([0,+\infty)\times\mathbb{R}\times\mathbb{R};\mathbb{R})$.
	\item [(F2)] There exists a constant $K\geq 0$ such that for all $(t,y,z)\in [0,+\infty)\times\mathbb{R}\times\mathbb{R}$
			\begin{displaymath}
				|f(t,y,z)|\leq K(1+|y|)(1+|z|).
			\end{displaymath}
	\item [(F3)] For all $(t,z)\in [0,+\infty)\times\mathbb{R}$, $y_1,y_2\in\mathbb{R}$
			\begin{displaymath}
				(y_1-y_2)(f(t,y_1,z)-f(t,y_2,z))\leq 0.
			\end{displaymath}
	\item [(F4)] There exist $L\geq 0$, $\alpha,\beta,\gamma\in (0,1]$ such that for all $t_1,t_2\in [0,+\infty)$, $y_1,y_2,z_1,z_2\in\mathbb{R}$
			\begin{eqnarray*}
				|f(t_1,y_1,z_1)-f(t_2,y_2,z_2)|&\leq& L\Bigl((1+|y_1|+|y_2|)\cdot(1+|z_1|+|z_2|)\cdot |t_1-t_2|^{\alpha}\notag\\
&&+|y_1-y_2|^{\beta}\notag\\
&&+(1+|z_1|+|z_2|)|y_1-y_2|^{\gamma}\notag\\
&&+(1+|y_1|+|y_2|)|z_1-z_2|\Bigr).
			\end{eqnarray*}
\end{itemize}
In the following fact we provide an example of the right-hand side function that satisfies the assumptions (F1)-(F4). In what follows we use this function in order to approximate the solution of \eqref{RD_0} in the case when $a_8\in (0,1)$, see also Remark \ref{remark_tilde_rho}. The proof of this fact is standard and we leave it to the reader.

    \begin{lem} 
    \label{tilde_f}
    Let the functions $A,B,C:[0,+\infty)\to [0,+\infty)$ satisfy H\"older condition with the H\"older exponent $\alpha\in (0,1]$ and with the H\"older constant $H\in [0,+\infty)$, $\varrho\in (0,1]$ and define a function $\tilde f:[0,+\infty)\times\mathbb{R}\times\mathbb{R}\to\mathbb{R}$ as follows\footnote{$sgn(x)=1$ if $x\geq 0$ and $sgn(x)=-1$ if $x<0$}
    	\begin{equation}
    		\tilde f(t,y,z)=A(t)-B(t)\cdot sgn(y)\cdot |y|-C(t)\cdot sgn(y)\cdot |y|^{\varrho}\cdot |z|.
    	\end{equation}
    	If the functions $A,B,C$ are bounded in $[0,+\infty)$, then the function $\tilde f$ satisfies (F1)-(F4) with $K=\|A\|_{\infty}+\|B\|_{\infty}+\|C\|_{\infty}$, $L=\max\{3H,\|B\|_{\infty},2\|C\|_{\infty}\}$, $\alpha=\alpha$, $\beta=1$ and $\gamma=\varrho$.
    \end{lem}
The following theorem is the main result of this section. It states the upper bound on th error of the Euler algorithm under the mild assumptions (F1)-(F4). We want to underline here that up to our knowledge there are no such results in the literature, since in standard situation at least the global Lipschitz condition is satisfied. Unfortunately, due to the form of the main equation \eqref{RD_0}, this condition is not satisfied, which supports necessity of the following result.
\begin{thm} 
\label{rate_of_conv_expl_Eul} Let $\eta\in\mathbb{R}$ and let $f$ satisfy (F1)-(F4). Fix $n\in\mathbb{N}$. Then there exist $C_0,C_1,\ldots,C_n\geq 0$ such that for sufficiently large $N\in\mathbb{N}$ the following holds
	\begin{equation}
		\max\limits_{0\leq k\leq N}|\phi_0(t_k^0)-y_k^0|\leq C_0 (h^{\alpha}+h^{\beta}+h^{\gamma}),
	\end{equation}
	and for $j=1,2,\ldots,n$
	\begin{equation}
		\max\limits_{0\leq k\leq N}|\phi_j(t_k^j)-y_k^j|\leq C_j(h^{1/2}+h^{\alpha}+h^{\beta}+h^{\gamma}),
	\end{equation}
	where $\phi_j=\phi_j(t)$ is the solution of (\ref{dde_1}) on the interval $[jt_{cr},(j+1)t_{cr}]$ and sequences $y_k^j, t_k^j$
		are calculated using explicit Euler method as described in \eqref{expl_euler_1}.
\end{thm}
We want to underline here that the theorem above gives the error estimates for the Euler scheme on the fixed and bounded time horizon $[0,(n+1)t_{cr}]$. This is crucial, since the discretization parameter $N$ depends on $n$.
 
In order to prove Theorem \ref{rate_of_conv_expl_Eul} we need several auxiliary lemmas. Note that the same symbol may be used for different constants. 
\begin{lem}
	\label{odes_exist_sol}
	Let us consider the  following ordinary differential equation
	\begin{equation}
	\label{ODE_1_Peano}
	z'(t)=g(t,z(t)), \quad t\in [a,b], \quad z(a)=\eta,
	\end{equation}
	where $-\infty<a<b<+\infty$, $\eta\in\mathbb{R}$ and $g:[a,b]\times\mathbb{R}\to\mathbb{R}$ satisfies the following conditions:
	\begin{itemize}
		\item [(G1)] $g\in C([a,b]\times\mathbb{R};\mathbb{R})$.
		\item [(G2)] There exists $K>0$ such that for all $(t,y)\in [a,b]\times\mathbb{R}$
		\begin{displaymath}
		|g(t,y)|\leq K(1+|y|).
		\end{displaymath}
		\item [(G3)] For all $t\in [a,b]$, $y_1,y_2\in\mathbb{R}$
		\begin{displaymath}
		(y_1-y_2)(g(t,y_1)-g(t,y_2))\leq 0.
		\end{displaymath}
	\end{itemize}	
	Then the equation \eqref{ODE_1_Peano} has a unique  $C^1$ solution in $[a,b]$,
	\begin{equation}
	\label{est_sol_z}
	\sup\limits_{t\in[a,b]}|z(t)|\leq (|\eta|+K(b-a))e^{K(b-a)},
	\end{equation}
	and for all $t,s\in [a,b]$
	\begin{equation}
	\label{lip_sol_z}
	|z(t)-z(s)|\leq \bar K |t-s|,
	\end{equation}
	where $\bar K=K\Bigl(1+(|\eta|+K(b-a))e^{K(b-a)}\Bigr)$.
\end{lem}
\begin{proof} Since the right-hand side function $g$ is continuous and it is of at most linear growth (i.e. (G1) and (G2) are satisfied), Peano's theorem guarantees existence of the solution
	(e.g. see Theorem 70.4, page 292 in \cite{goring}). 
	The uniqueness follows from the monotonicity condition (G3). Namely, let us assume that \eqref{ODE_1_Peano} has two solutions $z=z(t)$ and $x=x(t)$ with the same initial-value $z(a)=x(a)=\eta$. Then for all $t\in [a,b]$
	\begin{displaymath}	
	\frac{d}{dt}(z(t)-x(t))^2=2(z(t)-x(t))(g(t,z(t))-g(t,x(t)))\leq 0.
	\end{displaymath}
	Therefore, the mapping $[a,b]\ni t\mapsto (z(t)-x(t))^2$ is non-increasing and we get for all $t\in [a,b]$
	\begin{displaymath}
	(z(t)-x(t))^2\leq (z(a)-x(a))^2=0,
	\end{displaymath}
	which, together with continuity of $z,x$, implies that $z(t)=x(t)$ for all $t\in [a,b]$.
	
	For all $t\in [a,b]$ by (G2) we get
	\begin{equation}
	|z(t)|\leq |\eta|+\int\limits_a^t|g(s,z(s))|ds\leq |\eta|+K(b-a)+K\int\limits_a^t|z(s)|ds,
	\end{equation}
	and by Gronwall's lemma we obtain \eqref{est_sol_z}. The estimate \eqref{lip_sol_z} follows from (G2), \eqref{est_sol_z} and the mean value theorem. 
\end{proof}

The following result provides an upper bound on the error of explicit Euler method applied to ODEs with monotone and H\"older continuous right-hand side functions.
\begin{lem}
	\label{eul_odes}
	Let us consider the following ordinary differential equation
	\begin{equation}
	\label{ODE_1}
	z'(t)=g(t,z(t)), \quad t\in [a,b], \quad z(a)=\eta,
	\end{equation}
	where $-\infty<a<b<+\infty$, $\eta\in\mathbb{R}$ and $g:[a,b]\times\mathbb{R}\to\mathbb{R}$ satisfies the following conditions:
	\begin{itemize}
		\item [(G1)] $g\in C([a,b]\times\mathbb{R};\mathbb{R})$.
		\item [(G2)] There exists $K>0$ such that for all $(t,y)\in [a,b]\times\mathbb{R}$
		\begin{displaymath}
		|g(t,y)|\leq K(1+|y|).
		\end{displaymath}
		\item [(G3)] For all $t\in [a,b]$, $y_1,y_2\in\mathbb{R}$
		\begin{displaymath}
		(y_1-y_2)(g(t,y_1)-g(t,y_2))\leq 0.
		\end{displaymath}
		\item [(G4)] There exist $L>0$ and $\gamma_1,\gamma_2,\gamma_3\in (0,1]$  such that for all $t_1,t_2\in [a,b]$, $y_1,y_2\in\mathbb{R}$
		\begin{displaymath}
		|g(t_1,y_1)-g(t_2,y_2)|\leq L\Bigl((1+|y_1|+|y_2|)|t_1-t_2|^{\gamma_1}+|y_1-y_2|^{\gamma_2}+|y_1-y_2|^{\gamma_3}\Bigr).
		\end{displaymath}
	\end{itemize}	
	Let us consider the explicit  Euler method based on equidistant discretization. Namely, for $n\in\mathbb{N}$ we set $h=(b-a)/n$, $t_k=a+kh$, $k=0,1,\ldots,n$, and let $y_0\in\mathbb{R}$ be such that $|\eta-y_0|\leq\Delta$. We take
	\begin{equation}
		\label{def_Euler_g}
			y_{k+1}=y_k+h \cdot g(t_k,y_k), \quad k=0,1,\ldots,n-1.
	\end{equation}
	Then the following holds.
	\begin{itemize}
		\item [(i)] There exists $C_1=C_1(a,b,K,\eta)>0$ such that for all $n\in\mathbb{N},\Delta \in [0,+\infty)$  we have 
		\begin{equation}
		\label{eq:711}
			\max\limits_{0\leq k\leq n}|y_k|\leq C_1(1+\Delta).
		\end{equation}	
		\item [(ii)] There exists $C_2=C_2(a,b,L,K,\eta,\gamma_1,\gamma_2,\gamma_3)>0$ such that for all $n\in\mathbb{N},\Delta \in [0,+\infty)$  we have
		\begin{equation}
		\label{est_euler_err}
		\max\limits_{0\leq k\leq n}|z(t_k)-y_k|\leq C_2(\Delta+h^{\gamma_1}+h^{\gamma_2}+h^{\gamma_3}).
		\end{equation}
	\end{itemize}
\end{lem}
\begin{proof}  We have that
	\begin{equation}
	|y_{k+1}|\leq (1+hK)|y_k|+hK, \quad k=0,1,\ldots,n-1
	\end{equation}
	and $|y_0|\leq |\eta|+\Delta$. Hence, by the discrete version of Gronwall's lemma 
	we get that for all $k=0,1,\ldots,n$
	\begin{equation}
	\label{EST_yk}
	|y_k|\leq e^{K(b-a)}(|\eta|+\Delta+1)-1\leq C_1(1+\Delta),
	\end{equation}
	 where $C_1=\max\{e^{K(b-a)}(|\eta|+1)-1,e^{K(b-a)}\}$. This proves \eqref{eq:711}.
	
	For $k=0,1,\ldots,n-1$ we consider the following local ordinary differential equation
	\begin{equation}
		\label{def_local_zk}
			z_k'(t)=g(t,z_k(t)), \quad t\in [t_k,t_{k+1}], \quad z_k(t_k)=y_k.
	\end{equation}
	By (\ref{EST_yk}) we get for all $t\in [t_k,t_{k+1}]$ that
	\begin{equation}
	|z_k(t)|\leq |y_k|+\int\limits_{t_k}^t |g(s,z_k(s))|ds\leq C_1(1+\Delta)+K(b-a)+K\int\limits_{t_k}^t |z_k(s)|ds,
	\end{equation}
	and by the Gronwall's lemma we obtain
	\begin{equation}
	\label{norm_sup_zk}
	\sup\limits_{t\in [t_k,t_{k+1}]}|z_k(t)|\leq C_2(1+\Delta),
	\end{equation}
	where $C_2=(C_1+K(b-a))e^{K(b-a)}$. Therefore, for all $t\in [t_k,t_{k+1}]$
	\begin{equation}
	\label{est_zkyk}
	|z_k(t)-y_k|\leq\int\limits_{t_k}^t |g(s,z_k(s))|ds\leq hK\Bigl(1+\sup\limits_{t\in [t_k,t_{k+1}]}|z_k(t)|\Bigr)\leq C_3(1+\Delta) h,
	\end{equation}
	with $C_3=(1+C_2)K$. Now, we have that
	\begin{equation}
	\label{EST_INEQ1}
	|z(t_{k+1})-y_{k+1}|\leq |z(t_{k+1})-z_k(t_{k+1})|+|z_k(t_{k+1})-y_{k+1}|,
	\end{equation}
	for $k=0,1,\ldots,n-1$. Note that for all $t\in [t_k,t_{k+1}]$, due to the assumption (G3), the following holds 
	\begin{eqnarray}
	(z(t)-z_{k}(t))^2&=&(z(t_k)-y_k)^2+2\cdot\int\limits_{t_k}^t (z(s)-z_k(s))(g(s,z(s))-g(s,z_k(s)))ds\notag\\
		&\leq& (z(t_k)-y_k)^2.
	\end{eqnarray}
	Hence,  we arrive at
	\begin{equation}
	\label{diff_zzk}
	|z(t_{k+1})-z_k(t_{k+1})|\leq |z(t_k)-y_k|.
	\end{equation}
	We now estimate the second term in (\ref{EST_INEQ1}). We have by \eqref{def_Euler_g}, \eqref{EST_yk}, \eqref{def_local_zk}, \eqref{norm_sup_zk}, (G4), and \eqref{est_zkyk} that 
	\begin{eqnarray*}
	|z_k(t_{k+1})-y_{k+1}|&=&\Bigl|z_k(t_k)-y_k+\int\limits_{t_k}^{t_{k+1}}\Bigl(g(s,z_k(s))-g(t_k,y_k)\Bigr)ds\Bigl|\\
&\leq&\int\limits_{t_k}^{t_{k+1}}|g(s,z_k(s))-g(t_k,y_k)|ds
	\end{eqnarray*}
	\begin{displaymath}
	\leq L\int\limits_{t_k}^{t_{k+1}}\Bigl((1+|z_k(s)|+|y_k|)|s-t_k|^{\gamma_1}+|z_k(s)-y_k|^{\gamma_2}+|z_k(s)-y_k|^{\gamma_3}\Bigr)ds
	\end{displaymath}
	\begin{equation}
	\label{EST_INEQ21}
	\leq \tilde C_4 h \Bigl((1+\Delta)h^{\gamma_1}+(1+\Delta)^{\gamma_2}h^{\gamma_2}+(1+\Delta)^{\gamma_3}h^{\gamma_3}\Bigr),
	\end{equation}
	where $\tilde C_4=L\max\{(1+C_1+C_2)/(1+\gamma_1),C_3^{\gamma_2},C_3^{\gamma_3}\}$. Since $(1+\Delta)^{\gamma_i}\leq 2(1+\Delta)$, $i=1,2$, we obtain that 
	\begin{equation}
	\label{EST_INEQ2}
	|z_k(t_{k+1})-y_{k+1}|\leq C_4h(1+\Delta)(h^{\gamma_1}+h^{\gamma_2}+h^{\gamma_3}),
	\end{equation}
	where $C_4=2L\tilde C_4$. 

	Let us denote
	\begin{equation}
	e_k=z(t_k)-y_k, \quad k=0,1,\ldots,n.
	\end{equation}
	Of course $|e_0|\leq\Delta$. By \eqref{EST_INEQ1}, \eqref{diff_zzk}, and \eqref{EST_INEQ2} we have the following recursive inequality
	\begin{equation}
	|e_{k+1}|\leq |e_k|+C_4(1+\Delta) h (h^{\gamma_1}+h^{\gamma_2}+h^{\gamma_3}),
	\end{equation}
	for $k=0,1,\ldots,n-1$. It is easy to see that
	\begin{equation}
	|e_k|\leq \Delta+k C_4 (1+\Delta) h (h^{\gamma_1}+h^{\gamma_2}+h^{\gamma_3})\leq \Delta+C_4(b-a)(h^{\gamma_1}+h^{\gamma_2}+h^{\gamma_3}),
	\end{equation}
	\begin{equation}
	\leq C(\Delta+h^{\gamma_1}+h^{\gamma_2}+h^{\gamma_3}),
	\end{equation}
	for all $k=0,1,\ldots,n$, where $C=6(1+b-a)\max\{1,(b-a)C_4\}$. This ends the proof of \eqref{est_euler_err}, completing the proof of lemma. 
\end{proof}
In the following lemma we show, by using the results above, that the delay differential equation \eqref{dde_1} has unique solution under assumptions (F1)-(F3). Note that the assumptions are weaker than those known from the standard literature. Namely, we use only monotonicity and local H\"older condition for the right-hand side function $f$.
\begin{lem} 
\label{prop_sol_z}
Let $\eta\in\mathbb{R}$ and let $f$ satisfy (F1)-(F3). Then the equation \eqref{dde_1} has a unique continuously differentiable solution that exists for any $t\geq 0$. Moreover,
if we denote by $\phi_n=\phi_n(t)$ the solution of \eqref{dde_1} on the interval $[nt_{cr},(n+1)t_{cr}]$, then
for all $n\in\mathbb{N}_0$ there exist $K_0,K_1,\ldots,K_n\geq 0$ such that  
	\begin{equation}
		\sup\limits_{nt_{cr}\leq t\leq (n+1)t_{cr}}|\phi_n(t)|\leq K_n,
	\end{equation}
	and, for all $t,s\in [nt_{cr},(n+1)t_{cr}]$ 
	\begin{equation}
		|\phi_n(t)-\phi_n(s)|\leq \bar K_n |t-s|,
	\end{equation}
	with $\bar K_n=K(1+K_{n-1})(1+K_{n}),$ where $K_{-1}:=|\eta|$. 
\end{lem}
\begin{proof} We proceed by induction with respect to $n$.

For $n=0$, the equation \eqref{dde_1} can be written as
\begin{equation}
\label{dde_10}
	z'(t)=f(t,z(t),\eta), \quad t\in [0,t_{cr}],
\end{equation}
with the initial condition $z(0)=\eta$. Denoting by 
\begin{equation}
	g_0(t,y)=f(t,y,\eta), \quad t\in [0,t_{cr}], y\in\mathbb{R},
\end{equation}
we get, by the properties of $f$, that $g_0\in C([0,t_{cr}]\times\mathbb{R})$, 
\begin{equation}
	|g_0(t,y)|\leq\hat K_0(1+|y|),
\end{equation}
with $\hat K_0=K(1+|\eta|)$, and
\begin{equation}
	(y_1-y_2)(g_0(t,y_1)-g_0(t,y_2))\leq 0, \quad \textrm{ for all } y_1,y_2\in\mathbb{R}.
\end{equation}
Therefore, by Lemma \ref{odes_exist_sol} 
we get that there exists a unique continuously differentiable solution $\phi_0:[0,t_{cr}]\to\mathbb{R}$ of the equation \eqref{dde_10}, such that
\begin{displaymath}
	\sup\limits_{t\in [0,t_{cr}]}|\phi_0(t)|\leq K_0,
\end{displaymath}
where 
$$
K_0=(|\eta|+\hat K_0 t_{cr})e^{\hat K_0 t_{cr}}=(|\eta|+K(1+|\eta|)t_{cr})e^{K(1+|\eta|) t_{cr}}\geq 0,
$$ 
and for all $t,s\in [0,t_{cr}]$
\begin{displaymath}
	|\phi_0(t)-\phi_0(s)|\leq \bar K_0 |t-s|,
\end{displaymath}
where 
$$
\bar K_0=\hat K_0(1+K_0)=K(1+K_{-1})(1+K_0)
$$ depends only on $\eta,K,t_{cr}$. 

Let us now assume that there exists $n\in\mathbb{N}_{0}$ such that the statement of the lemma holds for the solution $\phi_n:[nt_{cr},(n+1)t_{cr}]\to\mathbb{R}$. Consider the equation
\begin{equation}
\label{dde_1np1}
	z'(t)=f(t,z(t),\phi_n(t-t_{cr})), \quad t\in [(n+1)t_{cr},(n+2)t_{cr}],
\end{equation}
with the initial condition $z((n+1)t_{cr})=\phi_n((n+1)t_{cr})$. Let
\begin{equation}
	g_{n+1}(t,y)=f(t,y,\phi_n(t-t_{cr})), \quad t\in [(n+1)t_{cr},(n+2)t_{cr}], y\in\mathbb{R}.
\end{equation}
We get by the inductive assumption and from the properties of $f$ that $g_{n+1}\in C([(n+1)t_{cr}],(n+2)t_{cr}]\times\mathbb{R};\mathbb{R})$, for all $y\in\mathbb{R}$ we have
\begin{equation}
	|g_{n+1}(t,y)|\leq K(1+\sup\limits_{nt_{cr}\leq t\leq (n+1)t_{cr}}|\phi_n(t)|)(1+|y|)\leq \hat K_{n+1}(1+|y|),
\end{equation}
with $\hat K_{n+1}=K(1+K_n)$, and
\begin{equation}
	(y_1-y_2)(g_{n+1}(t,y_1)-g_{n+1}(t,y_2))\leq 0, \quad y_1,y_2\in\mathbb{R}.
\end{equation}
Hence, again by Lemma \ref{odes_exist_sol} we get that there exists a unique continuously differentiable solution $\phi_{n+1}:[(n+1)t_{cr},(n+2)t_{cr}]\to\mathbb{R}$ of the equation \eqref{dde_1np1}, such that
\begin{displaymath}
	\sup\limits_{t\in [(n+1)t_{cr},(n+2)t_{cr}]}|\phi_{n+1}(t)|\leq K_{n+1},
\end{displaymath}
where 
$$
K_{n+1}=(K_n+\hat K_{n+1} t_{cr})e^{\hat K_{n+1} t_{cr}}=(K_n+K(1+K_n)t_{cr})e^{K(1+K_n)t_{cr}}\geq 0,
$$ 
and for all $t,s\in [(n+1)t_{cr},(n+2)t_{cr}]$ we have
\begin{displaymath}
	|\phi_{n+1}(t)-\phi_{n+1}(s)|\leq \bar K_{n+1} |t-s|,
\end{displaymath}
where $\bar K_{n+1}=\hat K_{n+1}(1+K_{n+1})=K(1+K_n)(1+K_{n+1})$. 

From the above inductive construction we see that the solution of \eqref{dde_1} is continuous. Moreover, due to the continuity of $f$, $\phi_n$, and $\phi_{n-1}$ we get  for any $n\in\mathbb{N}_0$ that
\begin{eqnarray}
\lim\limits_{t\to (n+1)t_{cr}-}z'(t)&=&\lim\limits_{t\to (n+1)t_{cr}-} \phi'_{n}(t)=\lim\limits_{t\to (n+1)t_{cr}-}f(t,\phi_n(t),\phi_{n-1}(t-t_{cr}))\notag\\
&=&f((n+1)t_{cr},\phi_n((n+1)t_{cr}),\phi_{n-1}(nt_{cr}))\notag\\
&=&f((n+1)t_{cr},\phi_{n+1}((n+1)t_{cr}),\phi_{n}(nt_{cr}))\notag\\
&=&\lim\limits_{t\to (n+1)t_{cr}+}f(t,\phi_{n+1}(t),\phi_{n}(t-t_{cr}))\notag\\
&=&\lim\limits_{t\to (n+1)t_{cr}+}\phi'_{n+1}(t)=\lim\limits_{t\to (n+1)t_{cr}+}z'(t).\notag
\end{eqnarray}
Hence, the solution of \eqref{dde_1} is continuously differentiable.

The proof is completed.
\end{proof}
\begin{lem} 
\label{props_gn}
Let $\eta\in\mathbb{R}$ and let $f$ satisfy (F1)-(F4). For any $n\in\mathbb{N}_0$ consider the function $g_n:[nt_{cr},(n+1)t_{cr}]\times\mathbb{R}\to\mathbb{R}$ given by
\begin{equation}
	g_n(t,y)=f(t,y,\phi_{n-1}(t-t_{cr})), 
\end{equation}
where $\phi_n=\phi_{n-1}(t)$ is the solution of (\ref{dde_1}) on the interval $[(n-1)t_{cr},nt_{cr}]$ and $\phi_{-1}(t):=\eta$ for all $t\in [-t_{cr},0]$.
Then
\begin{itemize}
	\item [(i)] $g_n\in C([nt_{cr},(n+1)t_{cr}]\times\mathbb{R}; \mathbb{R})$.
	\item [(ii)] For all $(t,y)\in [nt_{cr},(n+1)t_{cr}]\times\mathbb{R}$
		\begin{displaymath}
			|g_n(t,y)|\leq\hat K_{n}(1+|y|),
		\end{displaymath}
		where $K_{-1}=|\eta|$, $\hat K_n=K(1+K_{n-1})$ and $K_{n-1}$ is the constant from Lemma \ref{prop_sol_z}.
	\item [(iii)]	For all $t\in [nt_{cr},(n+1)t_{cr}], y_1,y_2\in\mathbb{R}$
	\begin{displaymath}
		(y_1-y_2)(g_n(t,y_1)-g_n(t,y_2))\leq 0.
	\end{displaymath}
	\item [(iv)] For all $t_1,t_2\in [nt_{cr},(n+1)t_{cr}], y_1,y_2\in\mathbb{R}$
	\begin{displaymath}
		|g_n(t_1,y_1)-g_n(t_2,y_2)|\leq \hat L_n \Bigl((1+|y_1|+|y_2|)\cdot |t_1-t_2|^{\alpha}+|y_1-y_2|^{\beta}+|y_1-y_2|^{\gamma}\Bigr),
	\end{displaymath}
	where $\hat L_n=L(1+2K_{n-1}+t_{cr}^{1-\alpha}\bar K_{n-1})$ and $\bar K_{-1}:=0$, $K_{-1}:=|\eta|$.
\end{itemize}
\end{lem}
\begin{proof} Conditions (i), (ii), and (iii) follow by Lemma \ref{prop_sol_z}. 

By the assumption (F4) and Lemma \ref{prop_sol_z} we get for all $t_1,t_2\in [nt_{cr},(n+1)t_{cr}]$, $y_1,y_2\in\mathbb{R}$ that
\begin{eqnarray*}		
	|g_n(t_1,y_1)-g_n(t_2,y_2)|&=&|f(t_1,y_1,\phi_{n-1}(t_1-t_{cr}))-f(t_2,y_2,\phi_{n-1}(t_2-t_{cr}))|\\
	&\leq& L\Bigl( (1+|y_1|+|y_2|)\cdot(1+|\phi_{n-1}(t_1-t_{cr})|\\
	&&+|\phi_{n-1}(t_2-t_{cr})|)\cdot|t_1-t_2|^{\alpha}\notag+|y_1-y_2|^{\beta}\\
	&&+(1+|\phi_{n-1}(t_1-t_{cr})|+|\phi_{n-1}(t_2-t_{cr})|)\cdot |y_1-y_2|^{\gamma}\\
	&&+(1+|y_1|+|y_2|)\cdot |\phi_{n-1}(t_1-t_{cr})-\phi_{n-1}(t_2-t_{cr})|\Bigr)\\
	&\leq& L\Bigl( (1+|y_1|+|y_2|)\cdot (1+2K_{n-1}) \cdot|t_1-t_2|^{\alpha}\notag\\
	&&+|y_1-y_2|^{\beta}+(1+2K_{n-1})\cdot |y_1-y_2|^{\gamma}\\
	&&+\bar K_{n-1}(1+|y_1|+|y_2|)\cdot |t_1-t_2|\Bigr)\\
	&\leq& \hat L_n\Bigl( (1+|y_1|+|y_2|)\cdot |t_1-t_2|^{\alpha}+|y_1-y_2|^{\beta}+|y_1-y_2|^{\gamma}\Bigr).
\end{eqnarray*}

\end{proof}

Now we are ready to prove Theorem \ref{rate_of_conv_expl_Eul}.
\\
\begin{proof}[Proof of Theorem \ref{rate_of_conv_expl_Eul}]
On the interval $[0,t_{cr}]$ we approximate the solution of (\ref{dde_1}) by the explicit Euler method 
\begin{eqnarray}
	y_0^0&=&\eta,\\
	y_{k+1}^0&=&y_k^0+h\cdot g_0(t_k^0,y_k^0), \quad k=0,1,\ldots, N-1,
\end{eqnarray}
where $g_0(t,y)=f(t,y,\eta)$. Applying Lemmas \ref{props_gn} and \ref{eul_odes} to $\eta:=\eta$, $g:=g_0$, $[a,b]:=[0,t_{cr}]$, $\Delta:=0$ we get that
\begin{equation}
\label{est_euler_0tcr}
		\max\limits_{0\leq k\leq N}|\phi_0(t_k^0)-y_k^0|\leq C_0 (h^{\alpha}+h^{\beta}+h^{\gamma}),
\end{equation}
where $C_0=C_0(t_{cr},L,K,\eta,\alpha,\beta,\gamma)\geq 0$, and
\begin{equation}
\label{est_euler_yk1}
	|y^0_k|\leq \tilde K_0, \quad k=0,1,\ldots,N,
\end{equation}
where $\tilde K_0=\tilde K_0(t_{cr},K,\eta)\geq 0$. 

In $[t_{cr},2t_{cr}]$ we consider the following differential equation
\begin{equation}
\label{dde_12}
	z'(t)=g_1(t,z(t)), \quad t\in [t_{cr},2t_{cr}],
\end{equation}
with the initial value $z(t_{cr})=\phi_0(t_{cr})=\phi_0(t^0_N)$ and $g_1(t,y)=f(t,y,\phi_0(t-t_{cr}))$. We approximate \eqref{dde_12} by the auxiliary Euler scheme 
\begin{eqnarray}
\label{aux_expl_euler_1}
	\tilde y_0^1&:=&y^{1}_0=y_N^0,\\
\label{aux_expl_euler_11}	
	\tilde y_{k+1}^1&=&\tilde y_k^1+h\cdot g_1(t_k^1,\tilde y_k^1), \quad k=0,1,\ldots, N-1.
\end{eqnarray}
By \eqref{est_euler_0tcr} we have that
\begin{equation}
	|\phi_1(t_0^1)-\tilde y_0^1|=|\phi_0(t_N^0)-y_N^0|\leq C_0(h^{\alpha}+h^{\beta}+h^{\gamma}).
\end{equation}
Applying Lemmas \ref{props_gn}, \ref{prop_sol_z} and \ref{eul_odes} to $\eta:=\phi_0(t^0_N)$, $g:=g_1$, $[a,b]:=[t_{cr},2t_{cr}]$, $\Delta:=C_0 (h^{\alpha}+h^{\beta}+h^{\gamma})$ we get that
\begin{equation}
	|\phi_0(t^0_N)|\leq K_0,
\end{equation}
and
\begin{equation}
\label{est_phi1_tyk1}
		\max\limits_{0\leq k\leq N}|\phi_1(t_k^1)-\tilde y_k^1|\leq \tilde C_1 (h^{\alpha}+h^{\beta}+h^{\gamma}),
\end{equation}
where $\tilde C_1=\tilde C_1(t_{cr},L,K,\eta,\alpha,\beta,\gamma)\geq 0$ that, in particular, depends on the initial value $\eta$ of the  equation \eqref{dde_1}. Let us denote by
\begin{displaymath}
	e_k^1=\tilde y_k^1-y_k^1, \quad k=0,1,\ldots,N,
\end{displaymath}
where we have that $e_0^1=\tilde y_0^1-y_0^1=0$. From \eqref{aux_expl_euler_11} and  \eqref{expl_euler_11} we have for $k=0,1,\ldots,N-1$ that
\begin{equation}
\label{rek_ek1}
	e_{k+1}^1=e_k^1+h\mathcal{R}^1_k+h\mathcal{L}^1_k,
\end{equation}
where
\begin{displaymath}
	\mathcal{R}^1_k=f(t_k^1,\tilde y_k^1,\phi_0(t_k^0))-f(t_k^1,y_k^1,\phi_0(t_k^0))
\end{displaymath}
and
\begin{displaymath}
	\mathcal{L}^1_k=f(t_k^1,y_k^1,\phi_0(t_k^0))-f(t_k^1,y_k^1,y_k^0).
\end{displaymath}
From \eqref{rek_ek1} we obtain that
\begin{displaymath}
	(e_{k+1}^1-h\mathcal{L}^1_k)^2=(e_k^1+h\mathcal{R}^1_k)^2,
\end{displaymath}
which, together with the assumption (F3), implies
\begin{equation}
	(e_{k+1}^1)^2\leq (e_{k}^1)^2+h^2(\mathcal{R}^1_k)^2+2he^{1}_{k+1}\mathcal{L}^1_k, \quad k=0,1,\ldots,N-1.
\end{equation}
Moreover, 
\begin{displaymath}
	e_{k+1}^1\mathcal{L}^1_k\leq\frac{1}{2}\Bigl((e^1_{k+1})^2+(\mathcal{L}^1_k)^2\Bigr),
\end{displaymath}
hence
\begin{equation}
	(e_{k+1}^1)^2\leq (e_{k}^1)^2+h^2(\mathcal{R}^1_k)^2+h(e^{1}_{k+1})^2+h(\mathcal{L}^1_k)^2, \quad k=0,1,\ldots,N-1.
\end{equation}
Since $0<1/(1-h)\leq 1+2h\leq 2$ for $h\in (0,1/2)$, we have that
\begin{equation}
\label{rec_eq_11}
	(e_{k+1}^1)^2\leq (1+2h)(e_{k}^1)^2+2h^2(\mathcal{R}^1_k)^2+2h(\mathcal{L}^1_k)^2, \quad k=0,1,\ldots,N-1.
\end{equation}
Recall a well known fact that for all $\varrho\in (0,1]$ and $x\in\mathbb{R}$ it holds
	\begin{equation}
		|x|^{\varrho}\leq 1+|x|.
	\end{equation}
Then by the assumption (F4), Lemma \ref{prop_sol_z} and \eqref{est_euler_0tcr} we have the following estimates
\begin{equation}
\label{rec_eq_12}
	|\mathcal{R}_k^1|\leq L\Bigl(|e_k^1|^{\beta}+(1+2\sup\limits_{0\leq t\leq t_{cr}}|\phi_0(t)|)\cdot |e_k^1|^{\gamma}\Bigr)\leq c_1(1+|e_k^1|),
\end{equation}
where $c_1=c_0(L,K_0)\geq 0$, and
\begin{equation}
\label{est_mathcal_L}
	|\mathcal{L}_k^1|\leq L(1+2|y_k^1|)\cdot |\phi_0(t_k^0)-y_k^0|\leq LC_0(1+2|y_k^1|)(h^{\alpha}+h^{\beta}+h^{\gamma}).
\end{equation}
Furthermore, it holds that for all $k=0,1,\ldots,N-1$
\begin{displaymath}
	|y_{k+1}^1|\leq |y_k^1|+h\cdot |f(t_k^1,y_k^1,y_k^0)|\leq |y_k^1|+hK(1+|y_k^1|)(1+|y_k^0|)
\end{displaymath}
\begin{displaymath}
\leq |y_k^1|+hK(1+\tilde K_0)(1+|y_k^1|)=(1+h\tilde c_1)|y_k^1|+h\tilde c_1,
\end{displaymath}
where $\tilde c_1=\tilde c_1(K,\tilde K_0)\geq 0$ and $|y^1_0|=|y_N^0|\leq \tilde K_0$. By using discrete Gronwall's inequality  we obtain
\begin{equation}
\label{est_euler_ykj}
	|y_k^1|\leq \tilde K_1, \quad k=0,1,\ldots,N.
\end{equation}
Therefore, by \eqref{est_euler_ykj} and \eqref{est_mathcal_L} we obtain for all $k=0,1,\ldots,N$
\begin{equation}
\label{rec_eq_13}
	|\mathcal{L}_k^1|\leq \bar c_1(h^{\alpha}+h^{\beta}+h^{\gamma}),
\end{equation}
with $\bar c_1$ independent of $N$. From \eqref{rec_eq_11}, \eqref{rec_eq_12}, and \eqref{rec_eq_13} we get for sufficiently large $N$ and $k=0,1,\ldots,N-1$ that
\begin{equation}
	(e_{k+1}^1)^2\leq (1+3h)(e_k^1)^2+D_1h^2+D_2h(h^{2\alpha}+h^{2\beta}+h^{2\gamma}),
\end{equation}
where $D_1,D_2\geq 0$ are independent of $N$. Solving this discrete inequality yields
\begin{equation}
\label{est_ek1_2}
	(e_k^1)^2\leq \bar D_1(h+h^{2\alpha}+h^{2\beta}+h^{2\gamma}), \quad k=0,1,\ldots,N,
\end{equation}
with $\bar D_1\geq 0$ independent of $N$. Since
\begin{displaymath}
	|\phi_1(t_k^1)-y_k^1|\leq |\phi_1(t_k^1)-\tilde y_k^1|+|e_k^1|,
\end{displaymath}
by \eqref{est_phi1_tyk1} and \eqref{est_ek1_2}, we arrive at
\begin{equation}
	\label{est_euler_1tcr}
		\max\limits_{0\leq k\leq N}|\phi_1(t_k^1)-y_k^1|\leq C_1 (h^{1/2}+h^{\alpha}+h^{\beta}+h^{\gamma}),
\end{equation}
with $C_1\geq 0$ independent of $N$. 

On the consecutive intervals we proceed by induction. Namely, let us assume that there exist $1\leq j\leq n-1$ and $C_j, \tilde K_j\geq 0$ such that
\begin{equation}
	\label{est_euler_jtcr}
		\max\limits_{0\leq k\leq N}|\phi_j(t_k^j)-y_k^j|\leq C_j (h^{1/2}+h^{\alpha}+h^{\beta}+h^{\gamma}),
\end{equation}
and 
\begin{equation}
	|y_k^j|\leq \tilde K_j, \quad k=0,1,\ldots,N.
\end{equation}
(For $j=1$  the statement has already been proven in \eqref{est_euler_1tcr} and \eqref{est_euler_ykj}.) In the interval $[(j+1)t_{cr}, (j+2)t_{cr}]$ we consider the following ODE
\begin{equation}
\label{dde_12J}
	z'(t)=g_{j+1}(t,z(t)), \quad t\in [(j+1)t_{cr}, (j+2)t_{cr}],
\end{equation}
with the initial value $z((j+1)t_{cr})=\phi_{j+1}((j+1)t_{cr})=\phi_j(t_0^{j+1})$ and $g_{j+1}(t,y)=f(t,y,\phi_j(t-t_{cr}))$. We approximate \eqref{dde_12J} by the following auxiliary Euler scheme 
\begin{eqnarray}
\label{aux_expl_euler_1J}
	\tilde y_0^{j+1}&:=&y^{j+1}_0=y_N^j,\\
	\tilde y_{k+1}^{j+1}&=&\tilde y_k^{j+1}+h\cdot g_{j+1}(t_k^{j+1},\tilde y_k^{j+1}), \quad k=0,1,\ldots, N-1.
\end{eqnarray}
By \eqref{est_euler_jtcr} we have that
\begin{equation}
	|\phi_{j+1}(t_0^{j+1})-\tilde y_0^{j+1}|=|\phi_j(t_N^j)-y_N^j|\leq C_j(h^{1/2}+h^{\alpha}+h^{\beta}+h^{\gamma}).
\end{equation}
Repeating the arguments used from \eqref{dde_12} to \eqref{est_euler_1tcr}, but now for $\eta:=\phi_j(t_0^{j+1})$, $g:=g_{j+1}$, $[a,b]:=[(j+1)t_{cr},(j+2)t_{cr}]$, $\Delta:=C_j (h^{1/2}+h^{\alpha}+h^{\beta}+h^{\gamma})$, we obtain
\begin{equation}
	|\phi_j(t_0^{j+1})|=|\phi_j(t_N^{j})|\leq K_j,
\end{equation}
and
\begin{eqnarray}
	\label{est_euler_j_1tcr}
		\max\limits_{0\leq k\leq N}|\phi_{j+1}(t_k^{j+1})-y_k^{j+1}|&\leq& \max\limits_{0\leq k\leq N}|\phi_{j+1}(t_k^{j+1})-\tilde y_k^{j+1}| + \max\limits_{0\leq k\leq N}|\tilde y_k^{j+1}-y_k^{j+1}| \notag\\
&\leq& C_{j+1} (h^{1/2}+h^{\alpha}+h^{\beta}+h^{\gamma}),
\end{eqnarray}
and
\begin{equation}
	|y_k^{j+1}|\leq \tilde K_{j+1}, \quad k=0,1,\ldots,N,
\end{equation}
with $\tilde K_{j+1},C_{j+1}\geq 0$ independent of $N$,  provided that $N$ is sufficiently large. This ends the proof. 
\end{proof}

\begin{rem}
\label{remark_tilde_rho}
In Lemma \ref{tilde_f} we provided an example of a function $\tilde f$ that satisfies (F1)-(F4). Note that $\tilde f$, with $A(t)=A_1(t)\cdot\dot{\varepsilon}(t), B(t)=A_2(t)\cdot\dot{\varepsilon}(t)^{1-a_9}, C(t)=A_3(t)$, coincides in $[t_{cr},+\infty)\times [0,+\infty)\times [0,+\infty)$ with
\begin{equation}
	f(t,y,z)=A(t)-B(t)\cdot y-C(t)\cdot y^{a_8}\cdot z, \quad t\geq t_{cr}, y\geq 0, z\geq 0,
\end{equation}
which is a right-hand side function of the main equation \eqref{RD_0} for $t\geq t_{cr}$ where $z$ represents the delay term $\mathcal{R}(t-t_{cr})$. Knowing that the solution $\rho$ of \eqref{RD_0} is unique and non-negative (see Theorem \ref{thm:exists}), the solutions of \eqref{RD_0} and 
\begin{equation}
	\tilde \rho'(t)=\tilde f(t,\tilde\rho(t),\tilde\rho(t-t_{cr})),
\end{equation}
coincide for $t\geq t_{cr}$. (We take $\tilde \rho(t)=\rho(t)$ for $t\in [-t_{cr},t_{cr}]$.) This justifies, why we can use the Euler scheme in order to  approximate the nonnegative solution $\rho$ of \eqref{RD_0}. 
\end{rem}
\section{Numerical experiments}\label{numerical_solutions}
In this section we will provide numerical simulations of \eqref{RD_0} with real world parameters. However, before we do so, we will take a closer look
to \eqref{RD_1}, that is we assume that $A_i$ and $\dot \eps$ are time independent. This will give us a preliminary insight into possible evolution of
\eqref{RD_0}. In particular we will see how it changes with the change of parameters $A_i$, in particular influence of the boundary condition $A_1/A_2$.
We will also check how the accuracy of solutions changes with time step, and comment of empirical speed of convergence of numerical solutions.

It is also worth mentioning that \eqref{RD_1}, while much simplified compared to \eqref{RD_0}, has its utility for modeling of our process. Namely, it can be used in inverse analysis,
as laboratory experiments are usually done in controlled environment. In particular $T$ and $\dot \eps$ can be assumed constant in laboratory experiments, which leads to time-independent coefficients $A_i(t)\equiv A_i$.
\subsection{Empirical tests of convergence - selected instances of equation \eqref{RD_1}}
As before, we divide our discussion into two cases when $a_8 \in \{ 0,1 \}$.

In the case when $\mathbf{a_8=0}$, we consider four sets of parameters (see Figure~\ref{fig:a80-a_n})
\begin{enumerate}[(i)]
	\item $\rho_0 = 0, \rho_{cr} = 1, A_1 = 10, A_2 = 1, A_3 = 0.9$,
	\item $\rho_0 = 0, \rho_{cr} = 4, A_1 = 10, A_2 = 2, A_3 =1$,
	\item $\rho_0 = 0, \rho_{cr} = 9, A_1 = 10, A_2 = 1, A_3 =0.9$.
\end{enumerate}
All solutions are considered on the interval $[0, 10 t_{cr}]$. Initial conditions on all particular intervals are given by values of corresponding analytical formula.
\begin{figure}
	\centering
	\begin{subfigure}[t]{0.8\textwidth}
		\centering
		\includegraphics[width=1\linewidth]{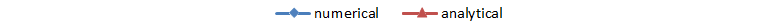}
	\end{subfigure}
	\begin{subfigure}[t]{.95\textwidth}
		\centering
		\includegraphics[width=1\linewidth]{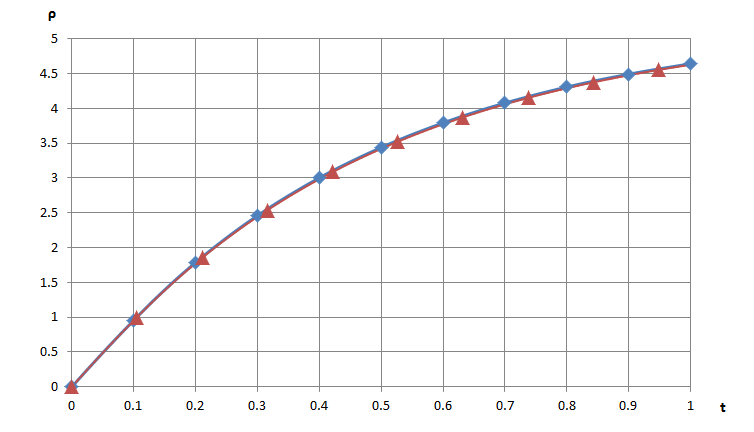}
		\caption{$a_8=0, \rho_0=0, \rho_{cr}=1, A_1=10, A_2=1, A_3=0.9$}\label{pic:a80-1-10-1-09_analytical}		
	\end{subfigure}
	\begin{subfigure}[t]{.95\textwidth}
		\centering
		\includegraphics[width=1\linewidth]{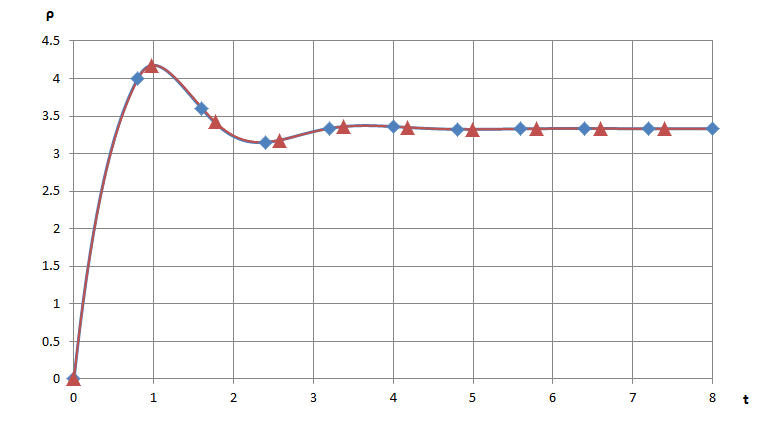}
		\caption{$a_8=0, \rho_0=0, \rho_{cr}=4, A_1=10, A_2=2, A_3=1$}\label{pic:a80-4-10-2-1_analytical}
	\end{subfigure}
	\begin{subfigure}[t]{.95\textwidth}
		\centering
		\includegraphics[width=1\linewidth]{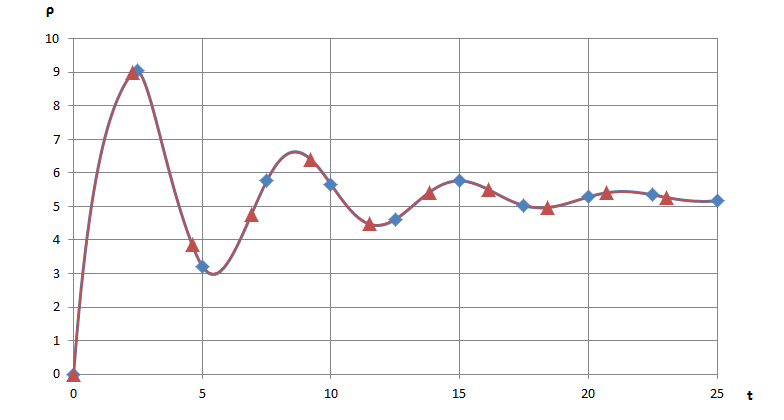}
		\caption{$a_8=0, \rho_0=0, \rho_{cr}=9, A_1=10, A_2=1, A_3=0.9$}\label{pic:a80-9-10-1-09_analytical}
	\end{subfigure}
	\caption{Exemplary analytical and numerical solutions in the case when $a_8=0$.}\label{fig:a80-a_n}
\end{figure}

We also considered case $\mathbf{a_8=0}$  with parameters range showing influence of violated condition $A_2/A_3$ (see Figure~\ref{fig:broken-assumption}):
\begin{enumerate}[(i)]
	\setcounter{enumi}{3}
	\item $\rho_0 = 0, \rho_{cr} = 9, A_1 = 10, A_2 = 1$, $A_3\in [0.9;1.5]$,
	\item $\rho_0=0, \rho_{cr}=4, A_1=10, A_2=1$, $A_3\in [0.5;5]$.
\end{enumerate}

\begin{figure}
	\centering
	\begin{subfigure}[t]{.95\textwidth}
		\setcounter{subfigure}{3}
		\centering
		\includegraphics[width=1\linewidth]{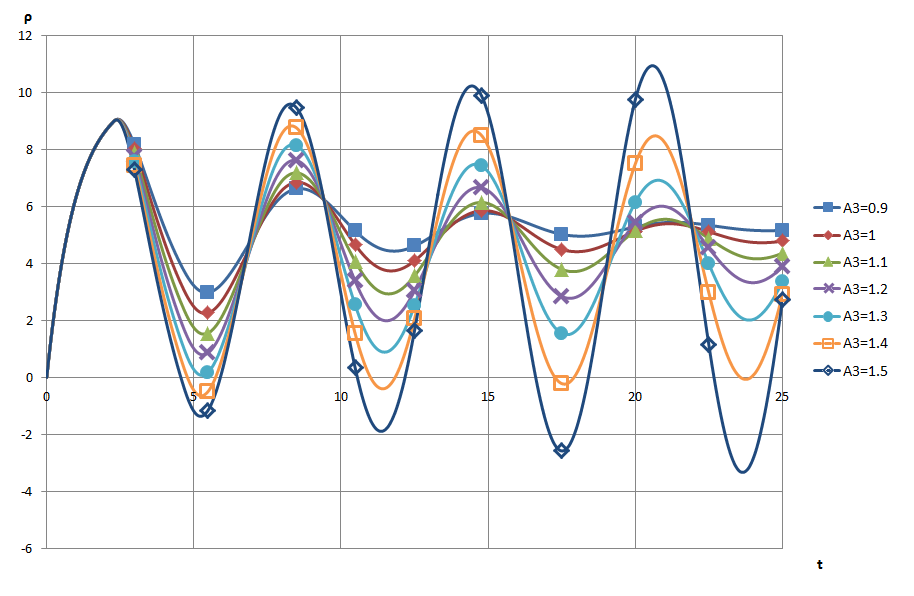}
		\caption{$a_8=0, \rho_0=0, \rho_{cr}=9, A_1=10, A_2=1$}\label{pic:a80-9-10-1-all}		
	\end{subfigure}
	\begin{subfigure}[t]{.95\textwidth}
		\centering
		\includegraphics[width=1\linewidth]{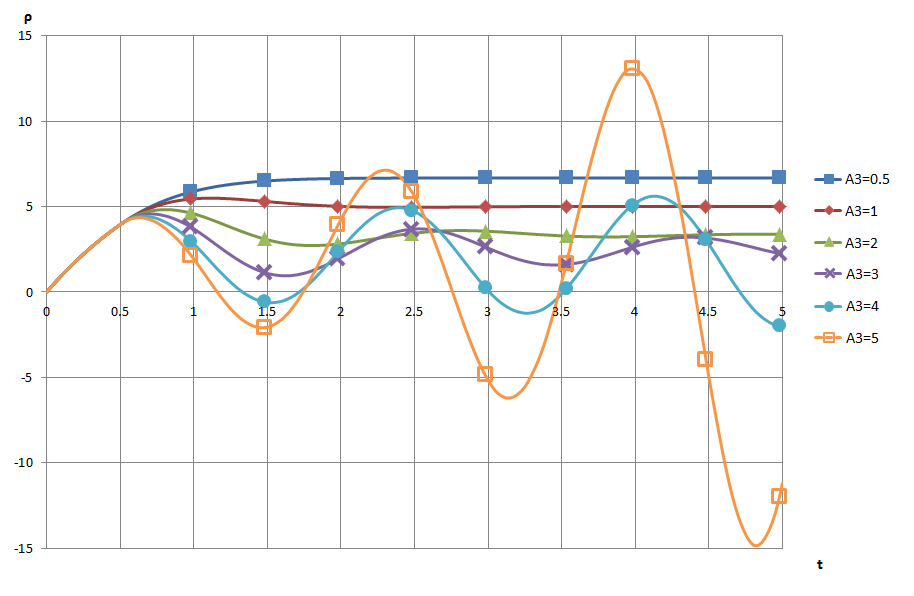}
		\caption{$a_8=0, \rho_0=0, \rho_{cr}=4, A_1=10, A_2=1$}\label{pic:a80-4-10-1-all}
	\end{subfigure}
	\caption{Exemplary numerical solutions in the case when some assumptions are not satisfied.}\label{fig:broken-assumption}
\end{figure}

Note, that in these cases $\frac{A_3}{A_2} = 1$ or we even have $\frac{A_3}{A_2} > 1$. Let us emphasize, that the assumptions of Theorem~\ref{thm:exists:a80} are broken. Nevertheless, the derived methods work properly what suggests that the assumptions might be weakened in further research. Notice however, that while solutions exists (and can be computed), it is hard to find their technological justification (recall that $\rho$ represents dislocation density, so Figure~\ref{fig:broken-assumption}(v) 
definitely cannot represent real technological process). Note that similarly to effect observed in \cite{UF} for equation similar to the case $a_8=0$,
	large $A_3/A_2>1$ (or large $t_{cr}$)  may lead to unbounded oscillations of solutions and technologically unjustified solutions. In fact, as we can see, such solutions may occur even when solution stabilizes (i.e. oscillations are bounded and decreasing in amplitude).

In the case when $\mathbf{a_8=1}$, we consider two exemplary sets of parameters (see Figure~\ref{fig:a81-numerical}):
\begin{enumerate}[(i)]\addtocounter{enumi}{5}
	\item $\rho_0 = 0, \rho_{cr} = 4, A_1 = 10, A_2 = 2, A_3 = 1$,
	\item $\rho_0 = 0, \rho_{cr} = 9, A_1 = 10, A_2 = 1, A_3 = 0.9$.
\end{enumerate}

\begin{figure}
	\setcounter{subfigure}{5}		
	\centering
	\begin{subfigure}[t]{.8\textwidth}
		\setcounter{subfigure}{5}
		\centering
		\includegraphics[width=1\linewidth]{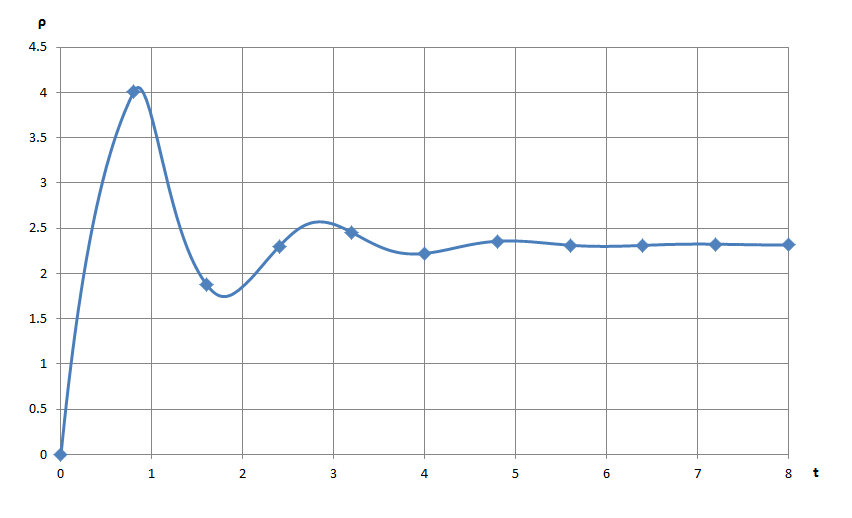}
		\caption{$a_8=1, \rho_0=0, \rho_{cr}=4, A_1=10, A_2=2, A_3=1$}\label{pic:a81-4-10-2-1_analytical}
	\end{subfigure}
	\begin{subfigure}[t]{.8\textwidth}
		\centering
		\includegraphics[width=1\linewidth]{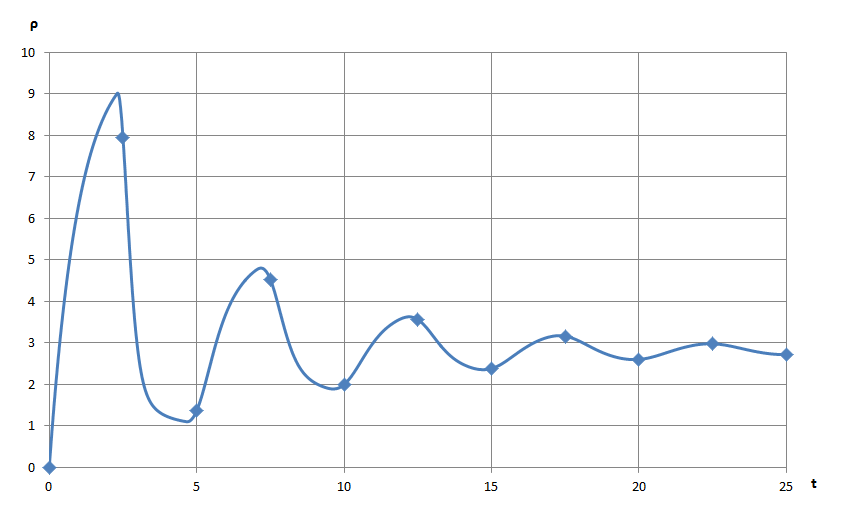}
		\caption{$a_8=1, \rho_0=0, \rho_{cr}=9, A_1=10, A_2=1, A_3=0.9$}\label{pic:a81-9-10-1-09_analytical}
	\end{subfigure}
	\caption{Exemplary numerical solutions in the case when $a_8=1$.}\label{fig:a81-numerical}
\end{figure}

Computing solution in consecutive intervals requires approximating of non elementary integrals. Therefore, the recursive formula for  analytical solution, even for the third interval $[2 t_{cr}, 3 t_{cr}]$, is computationally very demanding (as the integrals needs to be approximated independently in each iteration). 
Because of that, for error comparison using the analytical solution, we restrict our attention only to the interval $[0, 2 t_{cr}]$. Approximations by numerical methods, however, are computed for the whole considered interval $[0,10 t_{cr}]$. Then, the initial conditions for subsequent subintervals are taken from the numerical approximations.

In order to present numerical results of \eqref{RD_1} we have to introduce some additional notations. Let $\textbf{t}_N=(t_0^1, t_1^1, \ldots, t_N^1, \ldots, t_N^m)$ be a vector of points, where $t_k^n=nt_{cr}+kh,\ k=0,1,\ldots,N, \ n=1,2,\ldots,m$,  $h=\frac{t_{cr}}{N}$, and $m \in\mathbb{N}$. For given parameters of \eqref{RD_1}, we denote by $\textbf{z}_N=(z_0^1, z_1^1, \ldots, z_N^1, \ldots, z_N^m)$ a vector of values of exact solution of \eqref{RD_1} computed in $\textbf{t}_N$ points. Let  $\phi$ be explicit Euler, backward Euler or Runge-Kutta scheme, see \cite{Bellen}. For each set of parameters we tracked the behavior of the worst case error, estimated by
\begin{equation}\label{worst-case-error}
	\sup_{0 \leq k \leq N, 1 \leq n \leq m} |  \phi(t_k^n) - z_k^n |, \quad h=\frac{t_{cr}}{N},
\end{equation}
as $N \to \infty$. Results of numerical tests are presented for case (ii) in Figure~\ref{pic:a80_0-4-10-2-1-N} and Table~\ref{tab:a80_0-4-10-2-1}, and for case (vii) in Figure~\ref{pic:a81_0-9-10-1-09-N} and Table~\ref{tab:a81_0-9-10-1-09}, where points of consecutive iterations of the method are depicted by dots. As we can see, the case of $a_8=1$ requires more delicate analysis for choosing $N$, because too small number leads to having points not reflecting properly the dynamics of solutions.
\pic{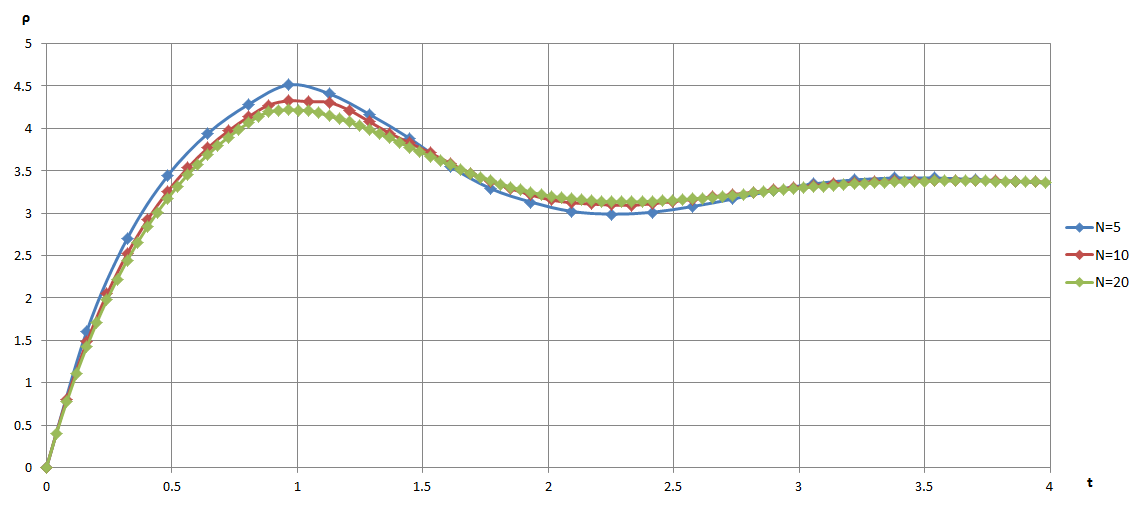}{1}{Numerical test for case (ii) $a_8=0, \\ \rho_0=0, \rho_{cr}=4, A_1=10, A_2=2, A_3=1$ by explicit Euler method. Dots represent consecutive steps of the method.}{pic:a80_0-4-10-2-1-N}
\tab{
explicit Euler &
$1.59 \cdot 10^{-1}$ &
$1.59 \cdot 10^{-1}$ &
$1.03 \cdot 10^{-1}$ &
$7.66 \cdot 10^{-1}$ &
$6.08 \cdot 10^{-2}$ \\ \hline
backward Euler & 
$1.39 \cdot 10^{-1}$ &
$1.39 \cdot 10^{-1}$ &
$9.44 \cdot 10^{-2}$ &
$7.16 \cdot 10^{-2}$ &
$5.77 \cdot 10^{-2}$ \\ \hline
RK4 &
$1.18 \cdot 10^{-5}$ &
$1.18 \cdot 10^{-5}$ &
$2.22 \cdot 10^{-6}$ &
$6.87 \cdot 10^{-7}$ &
$2.78 \cdot 10^{-7}$ \\ \hline
}{
explicit Euler &
$5.05 \cdot 10^{-2}$ &
$4.31 \cdot 10^{-2}$ &
$3.76 \cdot 10^{-2}$ &
$3.34 \cdot 10^{-2}$ &
$3.00 \cdot 10^{-2}$ \\ \hline
backward Euler &
$4.83 \cdot 10^{-2}$ &
$4.15 \cdot 10^{-2}$ &
$3.64 \cdot 10^{-2}$ &
$3.24 \cdot 10^{-2}$ &
$2.92 \cdot 10^{-2}$ \\ \hline
RK4 &
$1.33 \cdot 10^{-7}$ &
$7.12 \cdot 10^{-8}$ &
$4.15 \cdot 10^{-8}$ &
$2.58 \cdot 10^{-8}$ &
$1.69 \cdot 10^{-8}$ \\ \hline
}{
explicit Euler &
$1.49021416 \cdot 10^{-2}$ &
$1.48119132 \cdot 10^{-3}$ &
$1.48029707 \cdot 10^{-4}$ \\ \hline
backward Euler &
$1.47033478 \cdot 10^{-2}$ &
$1.47920577 \cdot 10^{-3}$ &
$1.48009854 \cdot 10^{-4}$ \\ \hline
RK4 &
$1.04235599 \cdot 10^{-9}$ &
$1.03390588 \cdot 10^{-13}$ &
$6.00819550 \cdot 10^{-16}$ \\ \hline
}{Error behavior for case (ii) $a_8=0, \rho_0=0, \rho_{cr}=4, \\ A_1=10, A_2=2, A_3=1$.}{tab:a80_0-4-10-2-1}

\pic{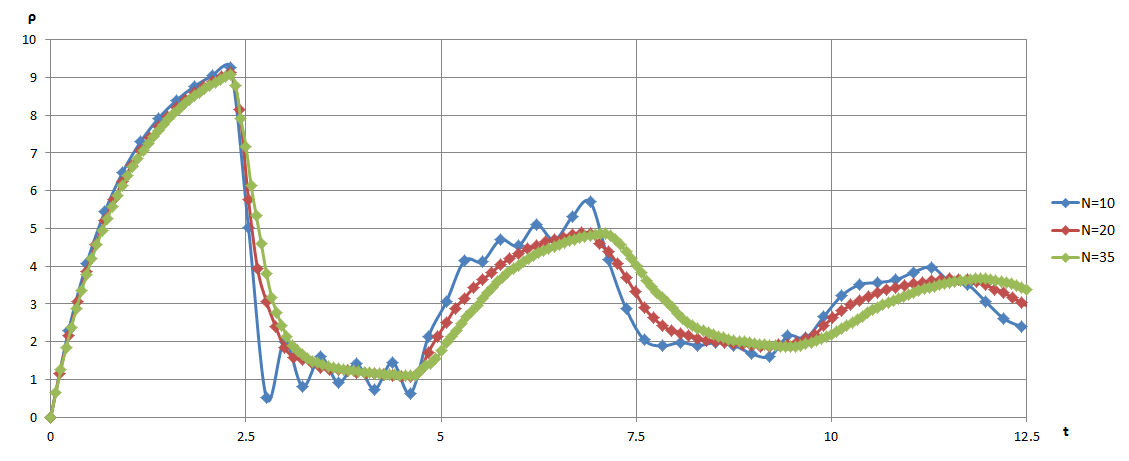}{1}{Numerical test for case (vii) $a_8=1,\\ \rho_0=0, \rho_{cr}=9, A_1=10, A_2=1, A_3=0.9$ by explicit  Euler method. Dots represent consecutive steps of the method.}{pic:a81_0-9-10-1-09-N}

\tab{
explicit Euler & 
$7.21$ &
$5.69$ &
$6.48$ &
$7.40$ &
$7.99$ \\ \hline
backward Euler & 
$6.68$ &
$5.06$ &
$6.48$ &
$7.40$ &
$7.99$ \\ \hline
RK4 &
$3.18 \cdot 10^{-2}$ &
$3.18 \cdot 10^{-2}$ &
$4.70 \cdot 10^{-3}$ &
$1.29 \cdot 10^{-3}$ &
$4.95 \cdot 10^{-4}$ \\ \hline
}{
explicit Euler & 
$7.83$ &
$7.48$ &
$6.22$ &
$3.39$ &
$2.48 \cdot 10^{-1}$ \\ \hline
backward Euler &
$7.82$ &
$7.41$ &
$6.12$ &
$3.54$ &
$2.31 \cdot 10^{-1}$ \\ \hline
RK4 &
$2.26 \cdot 10^{-4}$ &
$1.17 \cdot 10^{-4}$ &
$6.66 \cdot 10^{-5}$ &
$4.08 \cdot 10^{-5}$ &
$2.63 \cdot 10^{-5}$ \\ \hline
}{
explicit Euler &
$1.1936160 \cdot 10^{-1}$ &
$1.1449965 \cdot 10^{-2}$ &
$1.1403143 \cdot 10^{-3}$ \\ \hline
backward Euler &
$1.1448822 \cdot 10^{-1}$ &
$1.1398873 \cdot 10^{-2}$ &
$1.1398000 \cdot 10^{-3}$ \\ \hline
RK4 &
$1.5126559 \cdot 10^{-6}$ &
$1.4050918 \cdot 10^{-10}$ &
$1.5951269 \cdot 10^{-14}$ \\ \hline
}{Error behavior for case (vii) $a_8=1, \rho_0=0, \rho_{cr}=9, \\ A_1=10, A_2=1, A_3=0.9$.}{tab:a81_0-9-10-1-09}

\tab{
explicit Euler & 	
$3.51 \cdot 10^{-1}$ &
$3.51 \cdot 10^{-1}$ &
$2.31 \cdot 10^{-1}$ &
$1.72 \cdot 10^{-1}$ &
$1.37 \cdot 10^{-1}$ \\ \hline
backward Euler & 
$3.26 \cdot 10^{-1}$ &
$3.26 \cdot 10^{-1}$ &
$2.20 \cdot 10^{-1}$ &
$1.66 \cdot 10^{-1}$ &
$1.33 \cdot 10^{-1}$ \\ \hline
RK4 &
$6.86 \cdot 10^{-7}$ &
$6.86 \cdot 10^{-7}$ &
$1.33 \cdot 10^{-7}$ &
$4.15 \cdot 10^{-8}$ &
$1.69 \cdot 10^{-8}$ \\ \hline
}{
explicit Euler &
$1.14 \cdot 10^{-1}$ &
$9.77 \cdot 10^{-2}$ &
$8.54 \cdot 10^{-2}$ &
$7.58 \cdot 10^{-2}$ &
$6.82 \cdot 10^{-2}$ \\ \hline
backward Euler & 
$1.11 \cdot 10^{-1}$ &
$9.57 \cdot 10^{-2}$ &
$8.38 \cdot 10^{-2}$ &
$7.46 \cdot 10^{-2}$ &
$6.72 \cdot 10^{-2}$ \\ \hline
RK4 &
$8.11 \cdot 10^{-9}$ &
$4.36 \cdot 10^{-9}$ &
$2.55 \cdot 10^{-9}$ &
$1.59 \cdot 10^{-9}$ &
$1.04 \cdot 10^{-9}$ \\ \hline
}{
explicit Euler &
$3.39659300 \cdot 10^{-2}$ &
$3.38470334 \cdot 10^{-3}$ &
$3.38417599 \cdot 10^{-4}$ \\ \hline
backward Euler &
$3.37172869 \cdot 10^{-2}$ &
$3.38221709 \cdot 10^{-3}$ &
$3.38392736 \cdot 10^{-4}$ \\ \hline
RK4 &
$6.47237790 \cdot 10^{-11}$ &
$7.25560000 \cdot 10^{-15}$ &
$2.54110000 \cdot 10^{-15}$ \\ \hline
}{Error behavior for case (v) $a_8=0, \rho_0=0, \rho_{cr}=4, \\ A_1=10, A_2=1$ when $A_3=5$.}{tab:a80_0-4-10-1-5}

In all the cases where assumption $\frac{A_3}{A_2} < 1$ is satisfied, viz. (i)-(iii), (vi)-(vii) we observe the theoretical convergence rate. In the cases (iv)-(v) when ratio $\frac{A_3}{A_2}$ is slightly above $1$, some convergence to exact solutions can be observed. However, amplitude of oscillations increases with growing $\frac{A_3}{A_2}$, leading eventually to an unstable solution (see Figure~\ref{fig:broken-assumption}). Nonetheless, for a fixed number of intervals and tending with $N \ra \infty$, we still can observe the theoretical convergence rate of both Euler methods and Runge-Kutta scheme (see Table~\ref{tab:a80_0-4-10-1-5}).
\subsection{Equation \eqref{RD_0} with real world parameters}
In previous section we proved that explicit Euler method (and its modification) gives very satisfactory  results for simplified equation \eqref{RD_1} and proved stability of this method. This entitles us to perform numerical simulations on more complicated equation \eqref{RD_0} with both $\dot{\varepsilon}$ and $T$ time dependent and other parameters (time dependent as well) coming from real world processing of materials. First, let us examine solutions of \eqref{RD_1} with parameters established for DP steel and copper through inverse analysis for the experimental data (uniaxial compression tests performed at constant temperatures and strain rates) using algorithm described in \cite{Add102}. 
When equation \eqref{eq:sigma} is used to calculate the flow stress, the results which are depicted in Figure \ref{pic:real_world1}, are very similar to those in Figure \ref{pic:strain}. 

\begin{figure}[h!]	
	\centering
	\begin{subfigure}[t]{0.8\textwidth}
		\centering
		\includegraphics[width=1\linewidth]{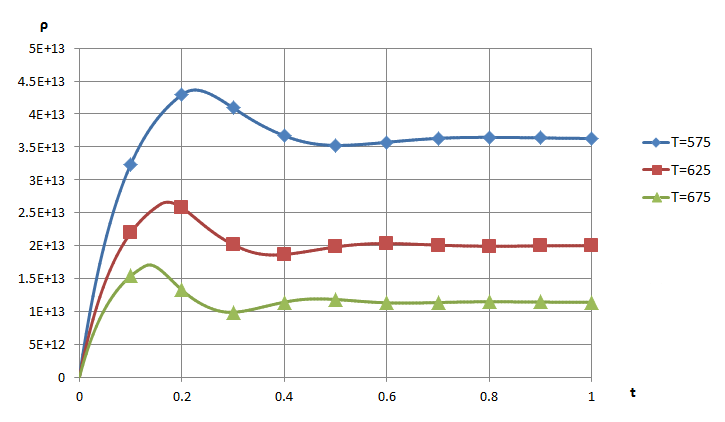}
		\caption{Coefficients for copper in cases: \\ T=575\textdegree C: $A_1=5.35882\cdot 10^{14}, A_2=11.134, A_3=9.9962\cdot 10^{-14}$, \\
T=625\textdegree C: $A_1=3.91516\cdot 10^{14}, A_2=12.9833, A_3=3.30145\cdot 10^{-13}$, \\
T=675\textdegree C: $A_1=2.95672\cdot 10^{14}, A_2=14.8963, A_3=9.61261\cdot 10^{-13}$.}\label{pic:copper_constantT}		
	\end{subfigure}
	\begin{subfigure}[t]{0.8\textwidth}
		\centering
		\includegraphics[width=1\linewidth]{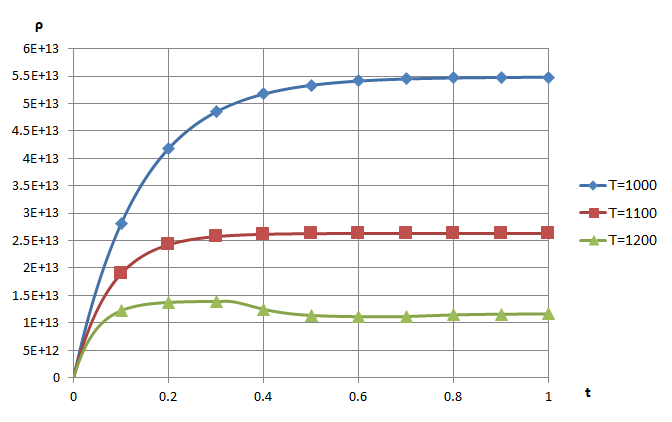}
		\caption{Coefficients for DP steel in cases: \\
T=1000\textdegree C: $A_1=3.93394\cdot 10^{14}, A_2=7.17277, A_3=6.41439\cdot 10^{-7}$, \\
T=1100\textdegree C: $A_1=3.34986\cdot 10^{14}, A_2=12.7284, A_3=1.49657\cdot 10^{-6}$, \\
T=1200\textdegree C: $A_1=2.91544\cdot 10^{14}, A_2=20.895, A_3=3.11231\cdot 10^{-6}$.}\label{pic:steel_constantT}
	\end{subfigure}
	\caption{Calculated solutions for constant temperature $T$ and strain rate $\dot{\varepsilon}=1$.}\label{pic:real_world1}
\end{figure}

This confirms that indeed, we are ready for modeling of real industrial process. Since it is characterized by strong heterogeneity of deformation, we decided to consider industrial process of hot strip rolling for demonstration of capabilities of the developed model. Similarly to previous laboratory case (see Figure~\ref{pic:real_world1}), two materials, DP steel and copper, were considered. Roll pass data, which were the same for both materials, were as follows: initial thickness 20 mm, thickness reduction 50\%, roll radius 400 mm and roll rotational velocity 10 rpm. Thermal-mechanical finite element (FE) model was used in the macro scale to calculate strains, stresses and temperatures. Details of the FE code are given in \cite{rolling,MP}. Briefly, the Levy-Mises flow rule was used as the constitutive law:

\begin{equation}
\sigma=\frac{2}{3} \frac{\sigma_f}{\dot{\varepsilon_i}} \dot{\varepsilon}
\end{equation}
where: $\sigma$, $\dot{\varepsilon}$ - stress and strain rate tensors, respectively, $\dot{\varepsilon_i}$ - effective strain rate, $\sigma_f$ - the flow stress provided by \eqref{eq:sigma}.

\pic{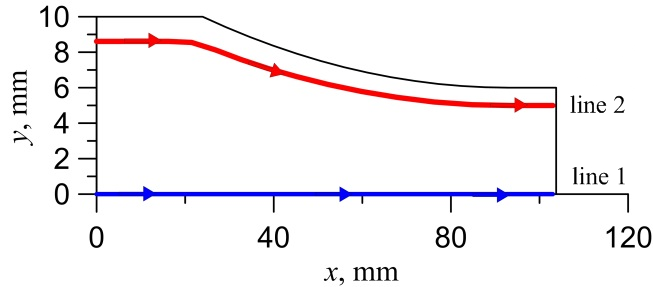}{0.6}{Shape of the deformation zone and flow lines along which the model was solved.}{pic:draw_roll_gap}

Equation \eqref{RD_0} was solved along the flow lines in the deformation zone using current local values of the strain rate and the temperature. The results for two lines, one in the center of the strip and the second one close to the surface (see Figure~\ref{pic:draw_roll_gap}), are presented on Figures~\ref{pic:real_world_copper} and \ref{pic:real_world_steel}. Let us explain the details behind these numerical experiments.
Due to horizontal symmetry only a top part of the roll gap is shown in Figures~\ref{pic:draw_roll_gap},~\ref{pic:maps_steel}, and~\ref{pic:maps_copper}.
The entry temperature was 1060\textdegree C for DP steel and 600\textdegree C for copper. Shear modulus $\mu$ was assumed to be time independent and equal $\mu = 45000$ MPa for copper and $\mu=75000$ MPa for DP steel.
For the assumed parameters the length of the computation domain was 105 mm and the time needed for the material point to flow through this domain was $0.28$ s. Finite element (FE) simulation of the rolling process was performed using FE code described in \cite{rolling,MP} and calculated distributions of the temperature, strain rate and strain are shown in Figure~\ref{pic:maps_steel}  and Figure~\ref{pic:maps_copper}. Results depicted in Figure~\ref{pic:maps_steel} are for the DP steel and in Figure~\ref{pic:maps_copper} for the copper.
On the basis of these results changes of the temperature and the strain rate along the flow lines in Figure~\ref{pic:draw_roll_gap} were determined, leading to time-dependent coefficients $A_1(t), A_2(t), A_3(t)$ as presented on  Figure~\ref{pic:real_world_copper} and Figure~\ref{pic:real_world_steel}. Another important coefficient is $a_8$
	as it is responsible for nonlinearity in \eqref{RD_0}. For copper it was possible to satisfactorily fit the model with $a_8=1$, however for DP steel it had to be fractional because $a_8=1$ was not leading to satisfactory fitting. Fitting was successful with $a_8=0.45239$ and this value was used in our simulations (cf. \cite{rolling}). By the same reason, the coefficient $a_9$ was set to $0$ for copper and to $0.13751$ for DP steel. The coefficient $a_{11}=10^4$ in both cases, which among other things, ensures that $\rho_{cr}$ is never $0$.

\begin{figure}	
	\centering
	\begin{subfigure}[t]{0.8\textwidth}
		\centering
		\includegraphics[width=0.8\linewidth]{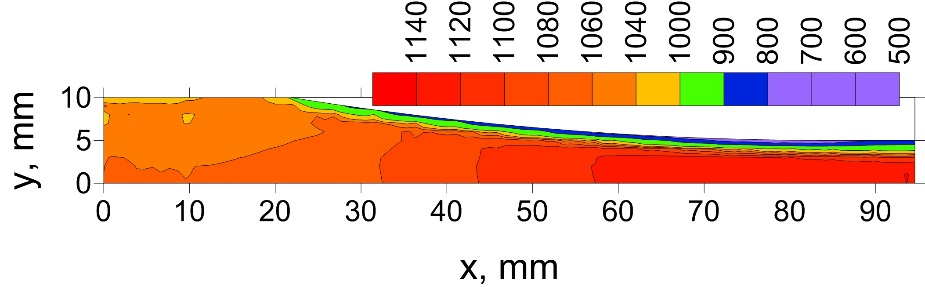}
		\caption{Distribution of the temperature $T$.}\label{pic:map_temp_steel}		
	\end{subfigure}
	\begin{subfigure}[t]{0.8\textwidth}
		\centering
		\includegraphics[width=0.8\linewidth]{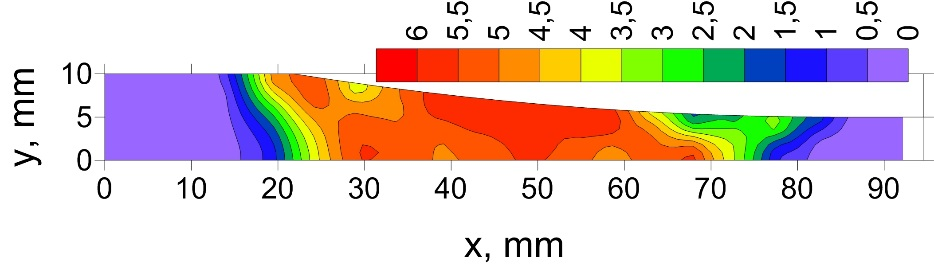}
		\caption{Distribution of the strain rate $\dot \varepsilon$.}\label{pic:map_strainrate_steel}
	\end{subfigure}
	\begin{subfigure}[t]{0.8\textwidth}
		\centering
		\includegraphics[width=0.8\linewidth]{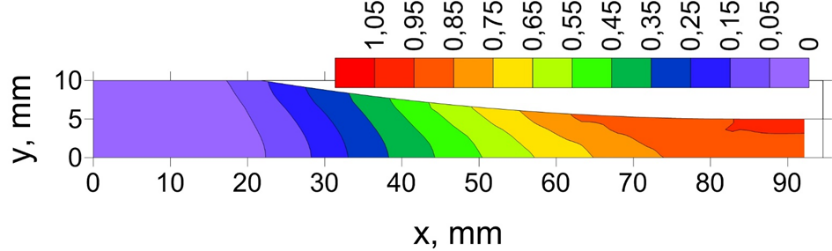}
		\caption{Distribution of the strain $\varepsilon$. }\label{pic:map_strain_steel}
	\end{subfigure}
	\caption{Calculated distribution of the temperature (i), the strain rate (ii) and the strain (iii) in the roll gap  for DP steel.}\label{pic:maps_steel}
\end{figure}

\begin{figure}	
	\centering
	\begin{subfigure}[t]{0.8\textwidth}
		\centering
		\includegraphics[width=0.8\linewidth]{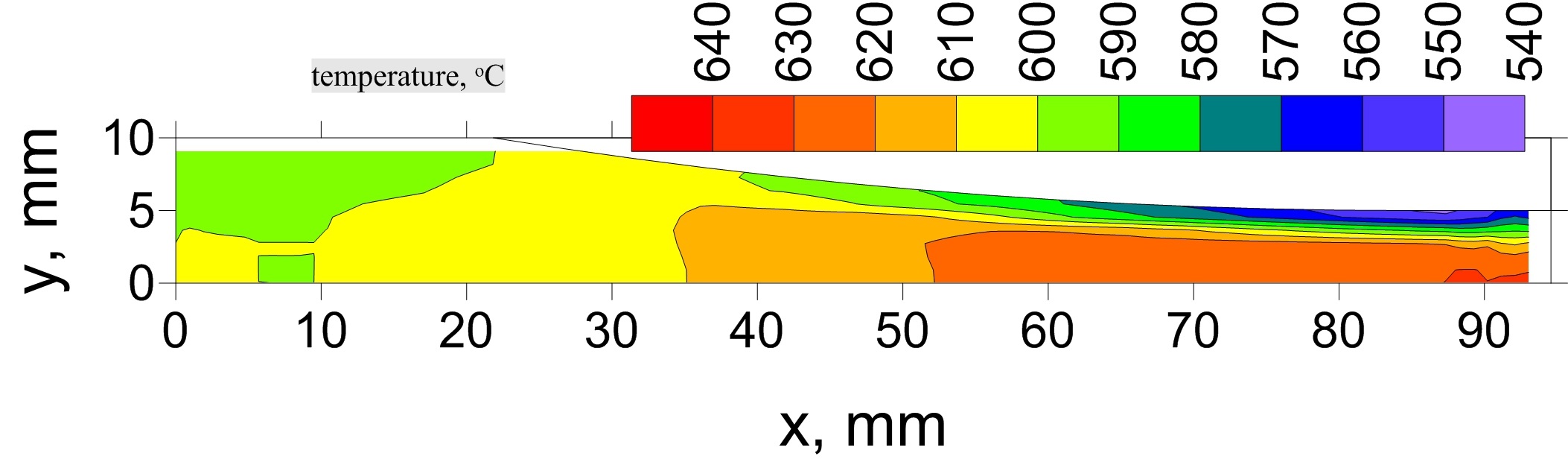}
		\caption{Distribution of the temperature $T$.}\label{pic:map_temp_copper}		
	\end{subfigure}
	\begin{subfigure}[t]{0.8\textwidth}
		\centering
		\includegraphics[width=0.8\linewidth]{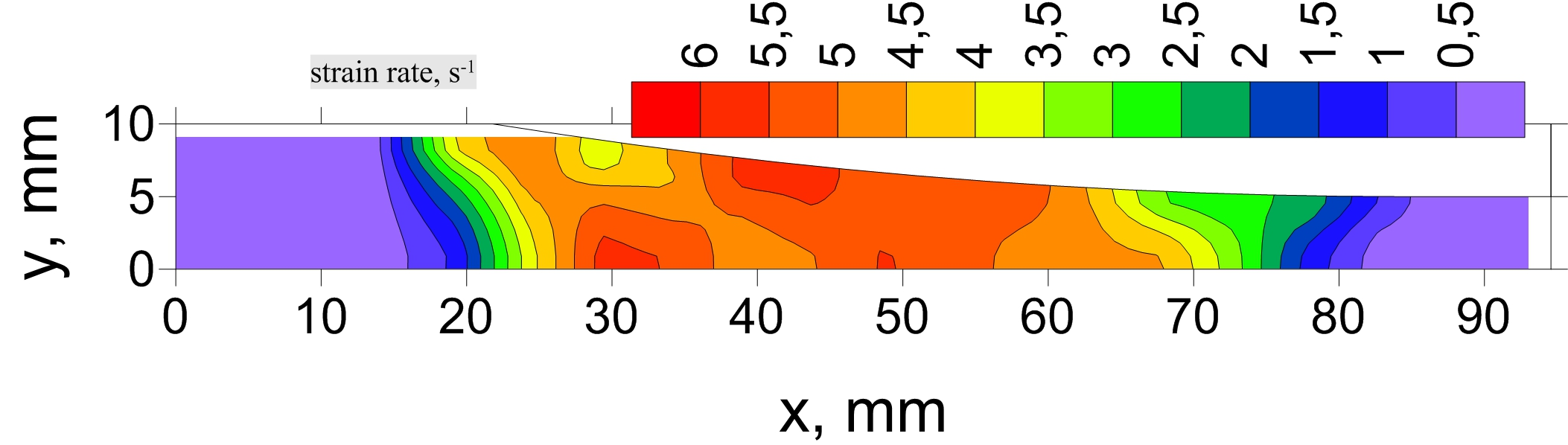}
		\caption{Distribution of the strain rate $\dot \varepsilon$.}\label{pic:map_strainrate_copper}
	\end{subfigure}
	\begin{subfigure}[t]{0.8\textwidth}
		\centering
		\includegraphics[width=0.8\linewidth]{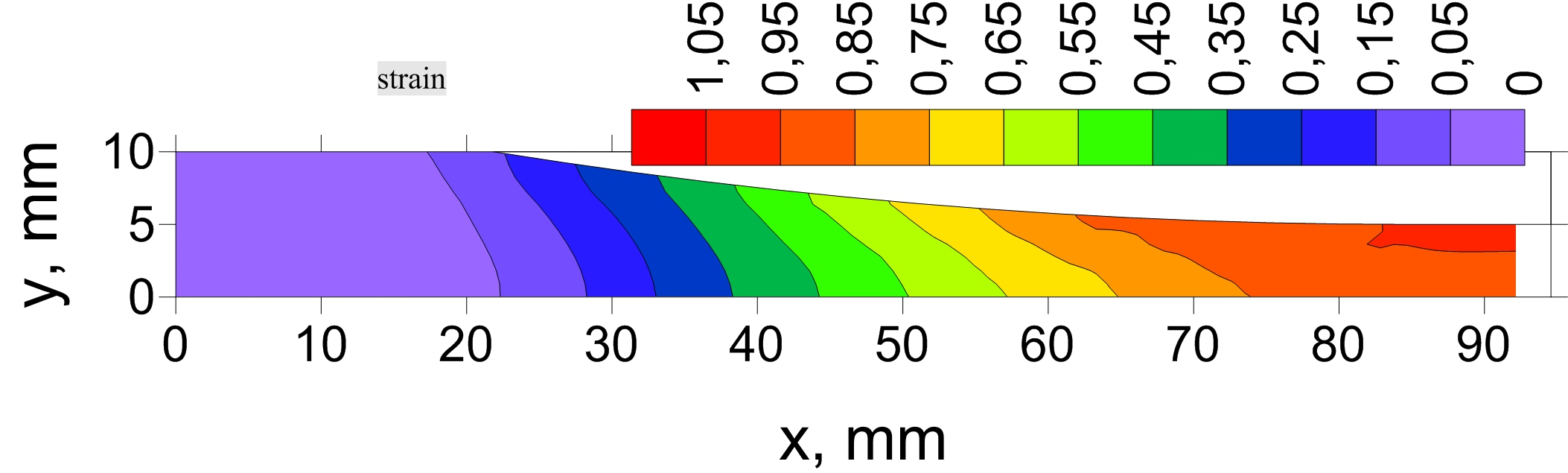}
		\caption{Distribution of the strain $\varepsilon$. }\label{pic:map_strain_copper}
	\end{subfigure}
	\caption{Calculated distribution of the temperature (i), the strain rate (ii) and the strain (iii) in the roll gap for copper.}\label{pic:maps_copper}
\end{figure}

We used this data to deal with \eqref{RD_0}. Calculated evolution solutions $\rho$ with parameters evolving along the lines 1 and 2 in Figure~\ref{pic:draw_roll_gap} are presented in Figure~\ref{pic:real_world_copper} and Figure~\ref{pic:real_world_steel}. Starting density $\rho_0$ was the same for both metals and equal $10^4 \textrm{ m}^{-2}$. Analysis of these results shows that they react properly to distinct temperature and strain rate histories for the center and surface areas. Practical observations show that in the center the temperature increases due to deformation heating. Contrary, drop of the temperature due to heat transfer to the cool roll is observed in the surface area. As far as strain rate is considered, in the central part it decreases monotonically due to monotonic deformation of this part. The results presented in Figure~\ref{pic:real_world_copper} and Figure~\ref{pic:real_world_steel} replicate properly material behavior in these conditions of the deformation. In the surface area, where the temperature is lower and the strain rate is higher, critical dislocation density is higher  and so $\rho_{cr}$ given by equation \eqref{eq:rhocr} is reached later. In the center of the strip higher temperature leads to more dynamic recrystallization and a decrease of the dislocation density.

\begin{figure}[h!]
	\centering
	\begin{subfigure}[t]{0.8\textwidth}
		\centering
		\includegraphics[width=1\linewidth]{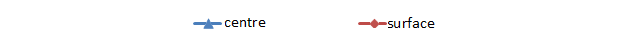}
	\end{subfigure}
	\centering
	\begin{subfigure}[t]{0.4\textwidth}
		\centering
		\includegraphics[height=4cm]{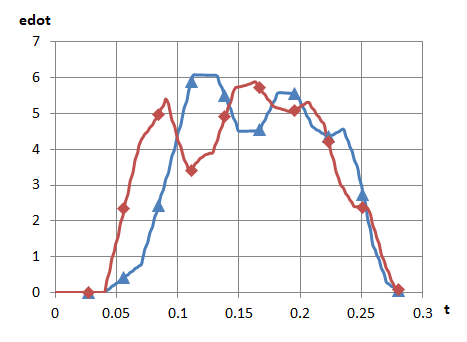}
		\caption{Graph of $\dot{\varepsilon}$}\label{pic:copper_edot}		
	\end{subfigure}
	\quad \quad
	\begin{subfigure}[t]{0.5\textwidth}
		\centering
		\includegraphics[height=4cm]{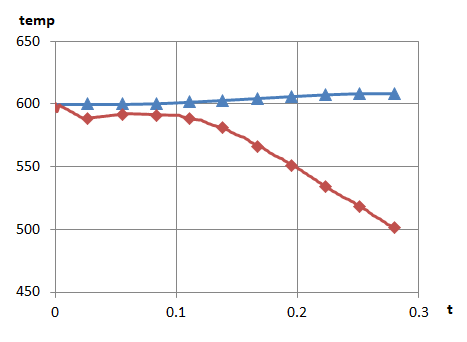}
		\caption{Graph of $T$}\label{pic:copper_temp}
	\end{subfigure}	
	\begin{subfigure}[t]{0.4\textwidth}
		\centering
		\includegraphics[height=4cm]{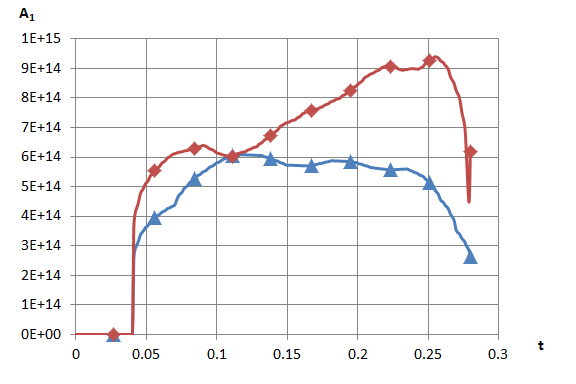}
		\caption{Graph of $A_1(t)$}\label{pic:copper_A1}		
	\end{subfigure}
	\quad
	\begin{subfigure}[t]{0.5\textwidth}
		\centering
		\includegraphics[height=4cm]{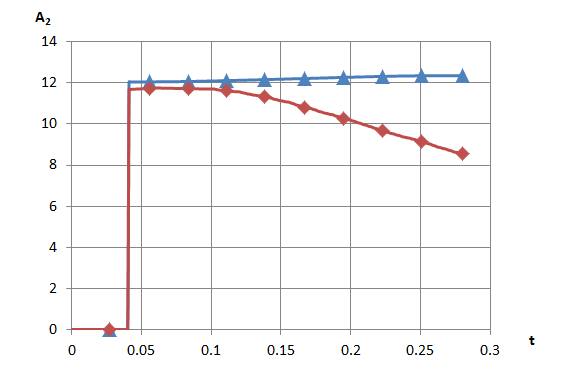}
		\caption{Graph of $A_2(t)$}\label{pic:copper_A2}
	\end{subfigure}
	\centering
	\begin{subfigure}[t]{0.4\textwidth}
		\centering
		\includegraphics[height=4cm]{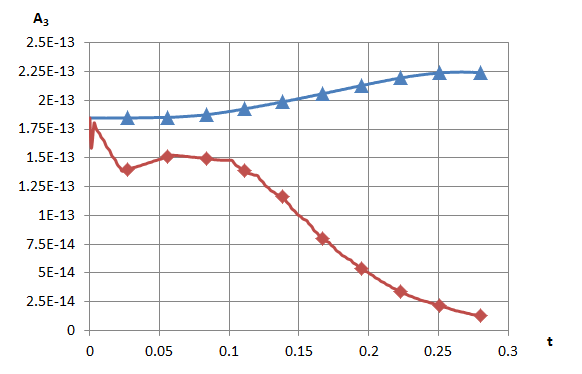}
		\caption{Graph of $A_3(t)$}\label{pic:copper_A3}		
	\end{subfigure}
	\quad \quad
	\begin{subfigure}[t]{0.5\textwidth}
		\centering
		\includegraphics[height=4cm]{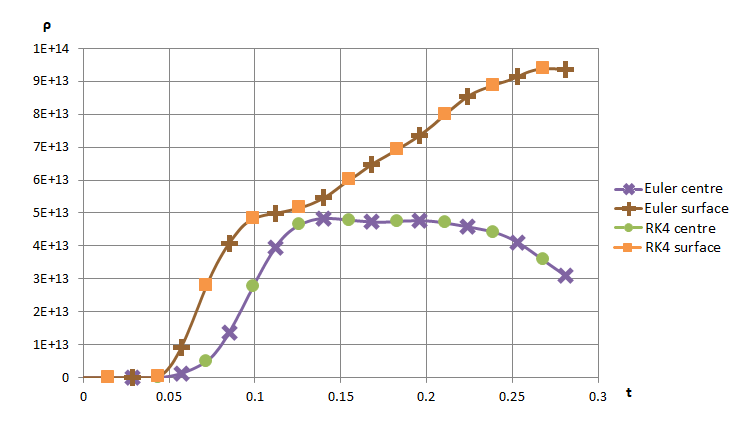}
		\caption{Solution $\rho(t)$ of \eqref{RD_0}}\label{pic:copper_rho}
	\end{subfigure}
	\caption{Calculated coefficients and solutions for copper. }\label{pic:real_world_copper}
\end{figure}

\begin{figure}[h!]
	\centering
	\begin{subfigure}[t]{0.8\textwidth}
		\centering
		\includegraphics[width=1\linewidth]{key1.png}
	\end{subfigure}
	\centering
	\begin{subfigure}[t]{0.4\textwidth}
		\centering
		\includegraphics[height=4cm]{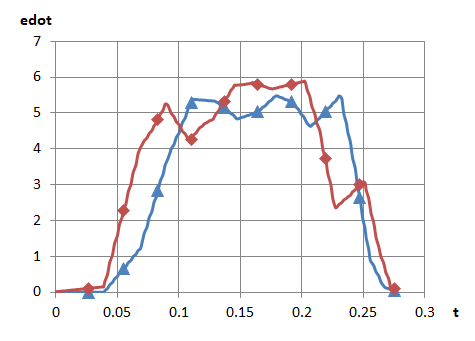}
		\caption{Graph of $\dot{\varepsilon}$}\label{pic:steel_edot}		
	\end{subfigure}
	\quad \quad
	\begin{subfigure}[t]{0.5\textwidth}
		\centering
		\includegraphics[height=4cm]{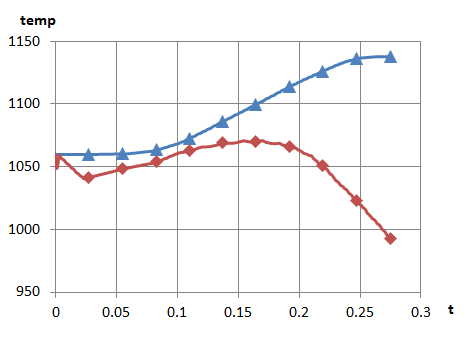}
		\caption{Graph of $T$}\label{pic:steel_temp}
	\end{subfigure}		
	\begin{subfigure}[t]{0.4\textwidth}
		\centering
		\includegraphics[height=4cm]{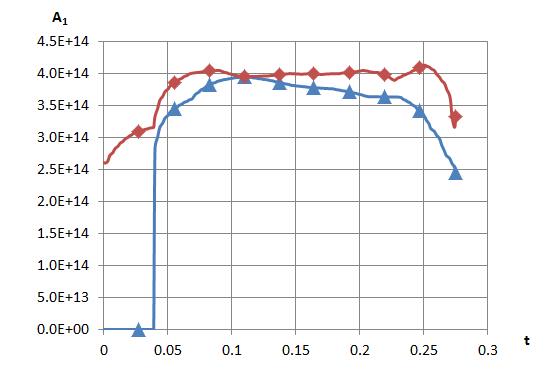}
		\caption{Graph of $A_1(t)$}\label{pic:steel_A1}		
	\end{subfigure}
	\quad
	\begin{subfigure}[t]{0.5\textwidth}
		\centering
		\includegraphics[height=4cm]{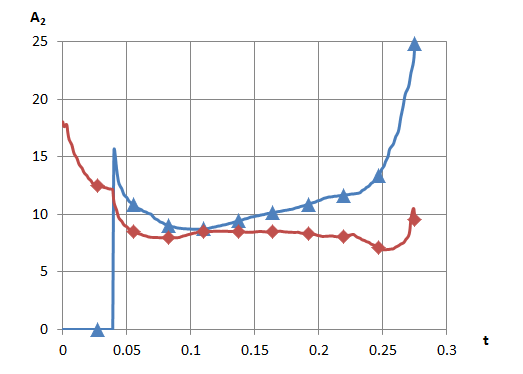}
		\caption{Graph of $A_2(t)$}\label{pic:steel_A2}
	\end{subfigure}
	\centering
	\begin{subfigure}[t]{0.4\textwidth}
		\centering
		\includegraphics[height=4cm]{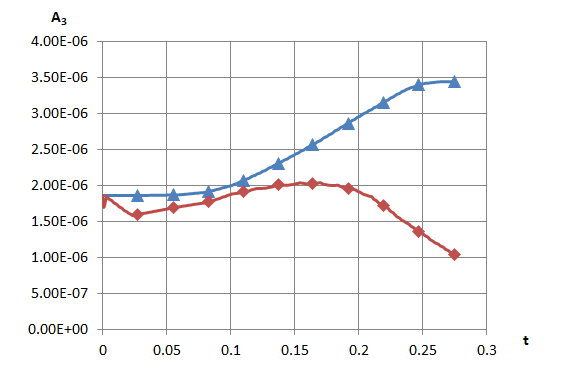}
		\caption{Graph of $A_3(t)$}\label{pic:steel_A3}		
	\end{subfigure}
	\quad \quad
	\begin{subfigure}[t]{0.5\textwidth}
		\centering
		\includegraphics[height=4cm]{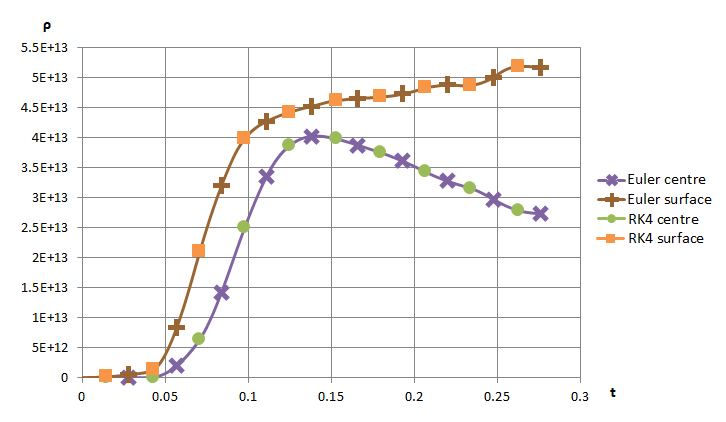}
		\caption{Solution $\rho(t)$ of \eqref{RD_0}}\label{pic:steel_rho}
	\end{subfigure}
	\caption{Calculated coefficients and solutions for DP steel.}\label{pic:real_world_steel}
\end{figure}
\section{Conclusions}
In the paper we have investigated mathematical aspects of  evolution of dislocation density in metallic materials, modeled by delay differential equations \eqref{RD_0}. For typical range of real world parameters we have shown that the unique solution always exists and it is bounded. For approximation of the solution we have used the explicit Euler method. We have shown the rate of convergence of the Euler method in the case when the right-hand side function is only locally H\"older continuous. We have confirmed our theoretical findings in numerical experiments performed in special cases, when explicit solutions were known. Moreover, we have applied the algorithm to examples with real-world parameters. Despite the fact that for a Runge-Kutta method we have not been able to investigate its error under conditions $(F1)-(F4)$, required by the equation, we tested its numerical behavior taking the Euler scheme as a benchmark. Numerical experiments showed advantage of the Runge-Kutta method over the Euler schemes. This encouraged us to use Runge-Kutta methods  in real world applications. However, investigation of its theoretical properties under the assumptions $(F1)-(F4)$ are forwarded to a future work. 
\section*{Acknowledgments}
This research was supported by National Science Centre (Narodowe Centrum Nauki - NCN) in Poland, grant no. 2017/25/B/ST8/01823.

\end{document}